\documentclass[12 pt,leqno]{amsart}

\usepackage[all]{xy}
\usepackage {amsfonts}
\usepackage{amsthm}
\usepackage{amssymb}
\usepackage{latexsym}
\usepackage{amsmath}
\usepackage{comment}
\usepackage{color}
\usepackage{a4wide}

\newenvironment{rlist}
{

\begin{enumerate}}
{\end{enumerate}}

\newtheorem{thmx}{Theorem}

\newtheorem{corx}[thmx]{Corollary}
\theoremstyle{plain}
\newtheorem{thm}{Theorem}[section]

\newtheorem {lem} [thm]{Lemma}
\newtheorem {prop}[thm] {Proposition}

\newtheorem{cor}[thm]{Corollary}

\theoremstyle{definition}
\newtheorem {defn}[thm] {Definition}
\newtheorem {rem} [thm]{Remark}
\newtheorem{ex}[thm]{Example}

\newcommand{\eqn}{\begin{equation}}
\newcommand{\eqne}{\end{equation}}

\DeclareMathOperator{\supp}{supp}
\DeclareMathOperator{\Bis}{Bis}

\newcommand{\Hilt}{\tilde{\Hil}}

\newcommand{\Ind}{\mathcal I}
\newcommand{\B}{\mathcal B}
\newcommand{\LL}{\mathcal L}

\newcommand{\G}{ G} 
\newcommand{\Gz}{{\G^{(0)}}} 

\newcommand{\Hil}{H}

\newcommand{\K}{\mathcal{K}}

\renewcommand{\L}{\mathcal L}

\newcommand{\A}{\mathcal A}

\newcommand{\C}{\mathbb C}
\newcommand{\R}{\mathbb R}
\newcommand{\Z}{\mathbb Z}
\newcommand{\N}{\mathbb N}
\newcommand{\T}{\mathbb T}
\newcommand{\bu}{\mathcal B}

\newcommand{\cst}{\ifmmode\mathrm{C}^*\else{$\mathrm{C}^*$}\fi}

\newcommand*{\onto}{\twoheadrightarrow}

\begin{document}
	\author{Bartosz K. Kwa\'sniewski}
	\address{Bartosz K. Kwa\'sniewski, Faculty of Mathematics, University of Bia\l ystok, ul. K. Cio\l kowskiego 1M, 15-245
	Bia\l ystok, Poland}
	\email{bartoszk@math.uwb.edu.pl}
	\author{Kang Li}
	\address{Kang Li, Institute of Mathematics of the Polish Academy of Sciences, ul.\ \'Sniadeckich 8, 00-656 Warszawa, Poland}
	\curraddr{Department of Mathematics, 
		Friedrich-Alexander-Universität Erlangen-N\"urnberg (FAU),
		Cauerstra\ss e 11, 91058 Erlangen, Germany}
	\email{kang.li@fau.de}
	\author{Adam Skalski}
	\address{Adam Skalski, Institute of Mathematics of the Polish Academy of Sciences, ul.\ \'Sniadeckich 8, 00-656 Warszawa, Poland}
	\email{a.skalski@impan.pl}

\title{\bf The Haagerup property for twisted groupoid dynamical systems}

\begin{abstract}

We introduce the Haagerup property for  twisted groupoid $C^*$-dynamical systems in terms of naturally defined positive-definite operator-valued multipliers. 
By developing a version of `the Haagerup trick' we prove that this property is equivalent to the  Haagerup property 
 of  the  reduced crossed product $C^*$-algebra with respect  to the canonical conditional expectation $E$. This extends a theorem of
Dong and Ruan for discrete group actions, and  implies that a given Cartan inclusion of separable $C^*$-algebras has the 
Haagerup  property if and only if the associated Weyl groupoid has the Haagerup property in the sense of Tu. We use the latter statement to prove that every separable $C^*$-algebra which has the Haagerup  property with respect to some Cartan subalgebra   satisfies the  Universal Coefficient Theorem. This  generalises a recent result of Barlak and Li on the UCT for nuclear Cartan pairs.

\end{abstract}



\date{\today} 
\maketitle

\section{Introduction}

One of the most active areas of research at the intersection of geometric group theory and the theory of operator algebras in recent years is the study of interactions between the geometric properties of a group and analytic properties of the associated $C^*$- or von Neumann algebras. 
The prototypical example is the equivalence of  amenability of a discrete group $\Gamma$ and the nuclearity of the group $C^*$-algebra $C^*_r(\Gamma)$ or injectivity of the group von Neumann algebra $VN(\Gamma)$, see e.g. \cite{BO}. 
Similarly, the Haagerup property  for $\Gamma$ can be characteristed  by approximation properties
of $C^*_r(\Gamma)$ or  $VN(\Gamma)$,  
see \cite{BO, book, choda, Dong}.
 The aforementioned correspondences, at a technical level,  are achieved in two stages: first one encodes some property of $\Gamma$ in terms of the existence of certain functions on $\Gamma$, and then one shows that an abstract approximation property of the associated operator algebra can be 
 achieved with the approximating maps being Herz-Schur multipliers, i.e.\ maps determined by functions on $\Gamma$. This second step is sometimes called `averaging maps into Herz-Schur multipliers', and is essentially based on a simple but very efficient trick attributed to Haagerup (see \cite{Haagerup}). 

With the increased interest in various  constructions generalising the group operator algebras there is a need to develop and apply appropriate versions of the above procedure in other contexts. This has been done,  for example, for (unimodular) discrete quantum groups (\cite{KrausRuan}, \cite{dfsw}) and for $C^*$-algebraic crossed products (\cite{mstt}, \cite{dong_ruan}). The latter case shows an interesting dichotomy: if one is interested directly in the Haagerup property of the crossed product by an action on a $C^*$-algebra $A$, one has to allow a broader class of analogues of Herz-Schur multipliers, sometimes called mapping multipliers (\cite{mtt}, \cite{BedosConti2}), whereas if one wants to consider only the $A$-valued multipliers, one arrives rather at a relative Haagerup property of the crossed product with respect to $A$ viewed as a subalgebra, as defined in \cite{dong_ruan}.

Recent years have also brought a renewed interest in $C^*$-algebras associated to twisted \'etale groupoids and their actions. This is partially due to the fact that such structures lead to many new interesting examples of $C^*$-algebras, but also due to a certain universality: the celebrated result of Renault from \cite{Re} says that any Cartan pair $(B,A)$, i.e. a $C^*$-algebra $B$ and a regular 
 maximal abelian subalgebra $A\subseteq B$ that admits a faithful conditional expectation, comes from a (unique) twisted-groupoid model. Also Cartan pairs are ubiquitous. For example  Li proves in \cite{XinLi}  that any simple classifiable $C^*$-algebra admits a Cartan subalgebra, so also a twisted groupoid model. 
On the other hand Barlak and Li show in  \cite{BL} that  nuclear algebras admitting Cartan subalgebras satisfy the UCT. 
They  exploit  results of Tu (\cite{Tu}) that connect the UCT property for \'etale groupoid $C^*$-algebras with what we call the Haagerup property of a groupoid. Our first initial motivation was  to generalise the result of \cite{BL} by replacing  nuclearity with a weaker property of a Cartan pair, equivalent to the Haagerup property of the associated Weyl groupoid. 
A relevant variant of the weaker property mentioned above was introduced by Dong and Ruan  \cite{dong_ruan}. 
The authors of \cite{dong_ruan}
defined the relative Haagerup property for a unital $C^*$-inclusion $A\subseteq B$ and characterised it when $B=A\rtimes_r \Gamma$ is 
the reduced crossed product, of  a discrete group action $\Gamma$ on a unital $C^*$-algebra $A$, via $Z(A)$-valued positive-definite  multipliers defined on $\Gamma$. Thus our second key motivation was to extend this result to the twisted actions of locally compact \'etale groupoids on not necessarily unital $C^*$-algebras. The necessity to consider the twist comes from the fact that Renault's theorem involves twisted groupoids.
However, it suffices to consider only \emph{Kumjian's twists} by a circle bundle \cite{Kumjian0}, 
and in the present paper we restrict ourselves  to such twists. Again a central role in the proofs is based on a suitable version of the Haagerup Trick.

In particular we establish the following two results (see Theorems \ref{FH}, \ref{uct for crossed product})  that are far reaching generalisations of  the main results of \cite{dong_ruan} and \cite{BL}, and  apply to $C^*$-algebras admitting  Cartan subalgebras (see Corollary \ref{GammaCartan pair}). Pertinent definitions of the  Haagerup  property can be found in Definitions \ref{def:Haagerup_for_twisted_actions} and \ref{defn:E-Haagerup} below.

\begin{thmx}\label{Theorem_A}
	Let $B$ be the reduced crossed product of a twisted action 
	$(\A,\G,\Sigma,\alpha)$ of an \'etale locally compact Hausdorff groupoid $\G$ on 
	the $C^*$-algebra $A:=C_0(\A)$ where $\A$ has a continuous unit section. Denote   the canonical expectation from $B$ to $A$ by $E$.
	Then  $B$ has the $E$-Haagerup property, if and only if $(\A,\G,\Sigma,\alpha)$ has the  Haagerup  property.
\end{thmx}

\begin{thmx}\label{Theorem_B}
	For any continuous separable saturated   Fell bundle $\B$ over a second countable locally compact Hausdorff \'etale groupoid $\G$ with the Haagerup property
	such that $C^*(\bu|_{\Gz})$  is of type I, both the reduced $C^*$-algebra $C^*_r(\bu)$ and the full $C^*$-algebra $C^*(\bu)$ satisfy the UCT.
\end{thmx}

\begin{corx}\label{Theorem_C}
	Let $B$ be a   separable $C^*$-algebra which admits a Cartan subalgebra $A$. If the inclusion $A \subseteq B$ has the Haagerup  property, then $B$ satisfies the UCT.
\end{corx}

An important fact which distinguishes our work is that we treat groupoids with not neccessarily compact unit spaces.
This forces us to work with non-unital algebras and has significant consequences even at the level of the relevant definitions.
Since the proof of Theorem \ref{Theorem_B}  relies heavily on the work of Tu (in fact it is a combination of the arguments of \cite{BL} and the Stabilization Theorem from \cite{stabilization}),  our model initial condition is 
existence of  a \emph{locally proper function of negative type} in the sense of  \cite{Tu}. 
This
local properness, when translated to an approximation property of positive-definite maps,  means 
that the maps in question are only \emph{locally $C_0$} (see Proposition \ref{Haagequivgeneral}). 
Accordingly, we  adapt the relative Haagerup property  of \cite{dong_ruan}  to not-neccesarily unital  $C^*$-inclusions $A\subseteq B$, mimicking the `locally $C_0$ condition' by using the Pedersen's ideal of $A$ (see Definition \ref{defn:E-Haagerup}).

As alluded to above, the key step in proving Theorem \ref{Theorem_A} is a version of \emph{the Haagerup Trick}
that gives a systematic way of passing from general maps acting on the crossed product  to  certain functions on $\G$, which in turn lead to analogues of Herz-Schur multipliers. In the context of operator algebras associated to groupoids, generalisations of Herz-Schur multipliers were studied for example in \cite{RamsayWalter}, \cite{RenaultFourier}, \cite{Takeishi} and \cite[Subsection 5.6]{BO}. 
We assume in Theorem \ref{Theorem_A} that $\A$ has a continuous unit section only to facilitate the Haagerup trick. This is enough 
for all of our applications and allows  us to consider very concrete and natural operator-valued multipliers. 
The genuinely new element is that we are dealing with the multipliers of \emph{twisted} groupoid $C^*$-dynamical systems,
with  twisting in the sense of Green-Renault  \cite{Green}, \cite{Renault}, \cite{Renault2}.
In the group case this brings us close to some considerations in \cite{BedosConti}.
The authors of \cite{BedosConti} employ  twists by cocycles  in the spirit of  \cite{Zeller-Meier}, \cite{BusbySmith}, rather than 
that of Green \cite{Green}, but as shown in \cite{PackerRaeburn}  cocycle twists cover  those coming from the group extension by a choice of 
a global continuous section.  
This   can not be directly generalised to groupoids, as   groupoid extensions admit non-trivial continuous sections only locally.
Nevertheless, it seems that we need a  variant of this result to make the technicalities related to our framework tractable.
We found a way out by describing  twisted groupoid dynamical systems  using   suitable pictures of  inverse semigroup actions, that exploit the ideas of \cite{BussExel}, \cite{BussMeyer}, \cite{BussExelMeyer} and
\cite{BartoszRalf2}. 
We establish the relevant equivalences, describing them in detail, as we believe they are of independent interest and could be used in other contexts. 
In particular, for  \'etale groupoids  twisted crossed products in the sense of Green-Kumjian-Renault can be 
replaced by inverse semigroup  crossed products twisted by a cocycle,  as introduced in \cite{BussExel}. This is the content of the following theorem (see Theorem \ref{thm:twisted_action_groupoid_vs_inverse_semigroup}).
\begin{thmx}\label{thm_D}
	Let $(\A,\G,\Sigma,\alpha)$ be a twisted groupoid $C^*$-dynamical system with $\G$ \'etale locally compact and Hausdorff.
	There is a natural  twisted inverse semigroup action $(A,S,\beta, \omega)$ 
	(described in Lemma \ref{lem:twisted_action_groupoid_vs_inverse_semigroup}) where $A=C^*(\A)$ and $S$ is an inverse semigroup consisting of
	open bisections, and we have natural isomorphisms
	for the corresponding crossed products:
	$
	C^*(\A,\G,\Sigma, \alpha)\cong A\rtimes_{\beta}^{\omega} S$,  $C^*_r(\A,\G,\Sigma, \alpha)\cong A\rtimes_{\beta, r}^{\omega} S.
	$
\end{thmx} 
We use the $S$-grading of the reduced crossed product from Theorem \ref{thm_D}
to establish our version of the Haagerup Trick and study its properties (see Propositions \ref{prop:Haagerup Trick}, \ref{prop:Hilbert multipliers_fell_groupoids}). We believe  this tool will be  useful in the study of other approximation properties of groupoid crossed products.

The detailed plan of the paper is as follows: after discussing preliminary facts and notations concerning groupoid Fell bundles in the beginning of Section \ref{Sec:prelim}, we provide in the same section an extended discussion of various equivalent pictures of twisted groupoid actions and associated crossed product operator algebras, focusing on the interpretation in terms of inverse semigroup actions on Hilbert bimodules.  In Section \ref{sect:Fell bundles} we introduce various classes of Herz-Schur type multipliers for such twisted crossed products and show that positive-definite $\A$-valued functions correspond to completely positive multipliers. Section \ref{Sect:HTrick} contains a version of the Haagerup Trick in our framework: every  map $\Phi$ on the resulting crossed product algebra which is a bimodule map (with respect to the subalgebra given by the $C^*$-bundle $\A$ on which our groupoid acts)  is shown to determine in a canonical way a function $h^\Phi$ on the groupoid with values in the associated bundle. There we also establish the interplay between properties of $\Phi$ and $h^\Phi$, which are key for the following results. In Section \ref{Haagerup property} we discuss the Haagerup property for groupoids, as defined in \cite{Tu}, and its variant (equivalent in the $\sigma$-compact setting) which provides the right notion for twisted groupoid dynamical systems. Then we exploit the content of the previous sections to show one of the main results of the paper, namely the equivalence between the Haagerup  property for such a system and a relative Haagerup property of the pair of associated $C^*$-algebras. We also record in this section certain immediate corollaries for Cartan pairs. Finally in Section \ref{Sect:UCT} we present some consequences of the theorem mentioned above for the UCT property of Cartan pairs, based on the observation that the key arguments of \cite{BL} require only the Haagerup property of the associated groupoid, and not neccessarily amenability, as was used in that paper.

\section{Groupoid Fell bundles, twisted groupoid actions and their $C^*$-algebras} \label{Sec:prelim}
In this section we  establish the notation and conventions regarding groupoids and their Fell bundles.
We also discuss the concept of twisted groupoid actions on $C^*$-algebras
and relate them to certain twisted inverse semigroup actions. Further we describe the relevant convolution operator algebras and prove their isomorphism (Theorem D of the introduction).

\subsection{Banach bundles}

Throughout this paper $X$ is a locally compact Hausdorff space.  
We consider Fell bundles that are upper semicontinuous, see for instance  \cite{dupre_gillette}, \cite{BussExel}.
So by a \emph{Banach bundle over} $X$  we mean 
a topological space $\bu=\bigsqcup_{x\in X }B_x$  such that the canonical projection
 $p:\bu\onto X$ is continuous and open,  each fiber $B_x$ is a complex Banach space, and the Banach space structure is 
consistent with the topology of $\bu$ in the sense that the maps $\bu\times \bu\supseteq \bigsqcup_{x\in X }B_x\times B_x\ni (v, w) \to v + w\in \bu$  and  
$\C\times \bu\ni (\lambda,v) \to \lambda v\in \bu $  are continuous,
 but  we require the map  $\bu\ni v\mapsto  \|v\|\in \R$  only to be upper  semicontinuous. One also requires another technical condition: if $\{v_i\}_{i \in \Ind}$ is a net in $\bu$ such that $\|v_i\|\stackrel{i \in \Ind}{\longrightarrow} 0$  and $p(v_i)\stackrel{i \in \Ind}{\longrightarrow} x$ in $X$, then $v_i$ converges to $0_x$, where $0_x$ is the zero element in $B_x$. 
If the map $\bu\ni v\mapsto  \|v\|\in \R$ is continuous,
 then we say the \emph{bundle is continuous}. We will also use an analogous notion of (upper semicontinuous) $C^*$-bundles. The bundle as above is said to be \emph{locally trivial} if for any point $x \in X$ there is an open neighbourhood $U$ of $x$ in $X$ and a Banach space $\mathsf{Y}$ such that $\bu|_{U} \cong U\times \mathsf{Y}$, where the isomorphism respects the bundle structure.

In a number of sources, including the original work of Fell,  Banach bundles are continuous  by definition,
see, for instance, \cite{Fell_Doran},  \cite{Kumjian}, \cite{Takeishi}. 
However,  nowadays it seems more standard to consider upper-semicontinuous  bundles. 
Firstly, this more general setting is often very useful. Secondly,
upper semicontinuity is usually  sufficient to prove important theorems for Banach bundles. 
For instance, the reconstruction theorem \cite[Theorem II.13.18]{Fell_Doran} is generalised (to upper semicontinuous bundles) 
in \cite[Proposition 2.4]{BussExel}, and the Douady-dal Soglio-Herault theorem \cite[Appendix C]{Fell_Doran}  is generalised in \cite[Corollary 2.10]{Lazar}.
The reconstruction theorem says that whenever we have a disjoint sum   of Banach spaces $\bu=\bigsqcup_{x\in X }B_x$ and a set  $\Gamma$ of sections of $\bu$ such that 
(a) for each $f\in \Gamma$ the function $X\ni x\mapsto \|f(x)\|\in \R$ is upper semicontinuous, and (b) for each $x\in X$ the
set $\{f(x):f\in \Gamma\}$ is dense in $B_x$, then  there is a unique topology on $\bu$ making it into a Banach bundle such that 
$\Gamma\subseteq C(\bu)$, i.e. all the sections in $\Gamma$ become continuous. The Douady-dal Soglio-Herault theorem states that for every 
Banach bundle $\bu=\bigsqcup_{x\in X }B_x$ and every $v\in B_x$ there is a continuous section $f\in C(\bu)$ such that $f(x)=v$. In fact,
 one can choose
$f$ to be vanishing at infinity in the sense that for every $\varepsilon >0$ the set $\{x\in X: \|f(x)\|\geq \varepsilon \}$ is compact.
In particular, the space of continuous sections vanishing at infinity, which we denote by  $C_0(\bu)$, is  a  Banach $C_0(X)$-module (with the norm $\|f\|=\sup_{x\in X}\|f(x)\|$),
and we can recover $\bu=\bigsqcup_{x\in X }B_x$ from this module.

A \emph{line bundle} is a Banach bundle $\bu=\bigsqcup_{x\in X }B_x$ where each space $B_x$, $x\in X$, is one-dimensional (we will usually use $\L$ to denote line bundles).
It is well known  that every  \emph{continuous line bundle is locally trivial}. More specifically, if $\bu$ is a continuous line bundle, then for every point $x_0\in X$ there is a section $f\in C_0(\bu)$ with $|f(x_0)|=1$ so that 
 $U:=\{x\in X: |f(x)| > 1/2\}$ is an open neighborhood of $x_0$. Then $C_0(\bu|_{U})\cong C_0(U)$ because for every $g\in C_0(\bu|_{U})$  
there is a unique function $h:U\to \C$  such that $g(x)=h(x)f(x)$ for every $x\in U$. Continuity of $g$ and $f$ forces continuity of $h$ 
and since $g\in C_0(\bu|_U)$ and $\|f|_U\|\geq 1/2$ we get $h\in C_0(U)$. Thus $\bu|_U=\bigsqcup_{x\in U}B_x$ is isomorphic to the trivial bundle.
We stress that general (upper semicontinuous) Banach bundles need not be locally trivial. 
For instance, we may consider  a bundle $\bu=\bigsqcup_{x\in X } \C$ with topology in which  all upper semicontinuous functions 
are continuous sections. 

\subsection{Fell bundles over groupoids and groupoid dynamical systems}
In this paper $\G$ will always denote an \'etale locally compact Hausdorff groupoid. Equivalently, \(\G\) is a topological groupoid such that 
 the unit space $X:=\G^{(0)}$ is locally compact and Hausdorff,  the source and range maps $s,r:\G\to X$ are local homeomorphisms
and $X$ is a clopen subset of $\G$ (note that it is in fact  automatically open as a range of a local homeomorphism).  
Yet another equivalent phrasing is that $\G$ is a locally compact Hausdorff groupoid 
with a topological basis consisting of bisections. 
A bisection for $\G$ is a subset (usually assumed to be open) $U\subseteq \G$ such that the restrictions $s:U\to s(U)$,   $r:U\to r(U)$ are 
homeomorphisms. The set of all open bisections of $\G$ will be denoted $\Bis(\G)$.

A \emph{Fell bundle over the groupoid~$\G$}
is (an upper-semicontinuous)  Banach bundle $\bu=\bigsqcup_{\gamma\in \G }B_\gamma$
equipped with a continuous involution
\(^*: \bu\to \bu\)
and a continuous multiplication
${\cdot}: \{(a,b)\in \bu\times \bu : a\in B_{\gamma_1},\ b
  \in B_{\gamma_2},\ (\gamma_1,\gamma_2)\in \G^{(2)}\} \to \bu$,
that satisfy the standard set of axioms,  for details see, for instance, 
 \cite[Section 2]{BussExel}  or \cite{Kumjian}, \cite{Takeishi}. 
 Accordingly, this structure
turns the fibres~\(B_x\)
for \(x\in X\)
into $C^*$-algebras
and the fibres~\(B_\gamma\)
for \(\gamma\in \G\)
into Hilbert $B_{r(\gamma)}$-$B_{s(\gamma)}$-bimodules,
such that the multiplication map yields isometric Hilbert bimodule
maps
$B_{\gamma_1}\otimes_{B_{s(\gamma_1)}}
B_{\gamma_2}\to B_{\gamma_1\gamma_2}
$
for all \((\gamma_1,\gamma_2)\in \G^{(2)}\).
If these maps are surjective, the Fell bundle~\(\bu\)
is called \emph{saturated}. This is equivalent to assuming that all 
Hilbert $B_{s(\gamma)}$-modules~\(B_\gamma\) are \emph{full}, i.e.\ for each $\gamma \in \G$
the linear span of $
\langle B_{\gamma}, B_{\gamma}\rangle_{B_{s(\gamma)}}:=\{\langle \xi, \eta\rangle_{B_{s(\gamma)}}:\xi, \eta \in B_{\gamma}\}
$
is dense in $B_{s(\gamma)}$. 
Obviously, all Fell line bundles are saturated.

For every Fell bundle $\bu=\bigsqcup_{\gamma\in \G }B_\gamma$ its restriction 
$\bu|_{X}:=\bigsqcup_{x\in X}B_x$  
to the unit space is an upper semicontinuous $C^*$-bundle and thus $A:=C_0(\bu|_{X})$ is a \emph{$C_0(X)$-$C^*$-algebra}, i.e. 
$A$ is a $C^*$-algebra and  a $C_0(X)$-module with the module map taking values in $ZM(A)$ - the center of the multiplier algebra.
Every $C_0(X)$-$C^*$-algebra $A$ arises as the algebra of continuous sections for some $C^*$-bundle $\bu|_{X}$, cf.\ \cite[Appendix C]{WLbook}; further $A$ is called a \emph{continuous $C_0(X)$-$C^*$-algebra} if the corresponding bundle is continuous.

(Twisted)
groupoid actions on (continuous) $C_0(X)$-$C^*$-algebras were introduced  in \cite{Renault}.  
\begin{defn} 
A \emph{groupoid $C^*$-dynamical system} is a triple $(\A,\G,\alpha)$  where $\A=\bigsqcup_{x\in X}A_x$ is an upper semicontinuous $C^*$-bundle,
$\G$ is a topological groupoid with  unit space $X$ 
and $\alpha$ 
is a \emph{continuous groupoid homomorphism}, i.e.\ 
a family $\{\alpha_{\gamma}\}_{\gamma\in \G}$ of $*$-isomorphisms $\alpha_\gamma:A_{s(\gamma)}\to  A_{r(\gamma)}$, $\gamma\in \G$, such 
that $\alpha_{\gamma\eta}=\alpha_{\gamma}\circ\alpha_{\eta}$ for all $(\gamma,\eta)\in \G^{(2)}$
and the map $(\gamma,a) \to \alpha_{\gamma}(a)$ from 
 $\G *_s \A :=\{ (\gamma,a)\in \G \times \A:  a\in A_{s(\gamma)}\}$ to $\A$   is continuous.
\end{defn}
\begin{rem}
Sometimes groupoid  $C^*$-dynamical systems  are written as $(A,\G,\alpha)$
where $A=C_0(\A)$ is the  $C_0(X)$-$C^*$-algebra corresponding to the $C^*$-bundle $\A$.
If $A$ is (isomorphic to) $C_0(X)$, that is $\A=X\times \C$ is a trivial line bundle, then 
there is a unique action $\alpha$ of $\G$ on $\A$ consisting of identities on $\C$. Thus
 the triple $(C_0(X), \G, \alpha)$ might be then identified with $\G$. 
Note also that for every groupoid $C^*$-dynamical system the `unit' automorphisms $\alpha_x\in \textup{Aut}(A_x)$ for $x \in \Gz$ are all trivial. 
 
\end{rem}
\begin{ex}[\cite{Kumjian}]\label{ex:bundle_from_action}
Every groupoid $C^*$-dynamical system $(\A,\G,\alpha)$ yields a saturated Fell bundle $\A\rtimes_{\alpha} \G$ over $\G$ that extends
the $C^*$-bundle $\A$. 
As a topological space $\A\rtimes_{\alpha} \G$ is the pullback bundle  subspace $\A*_{r} \G:=\{(a,\gamma)\in \A\times \G: a\in A_{r(\gamma)}, \gamma \in \G\}$.
We may also write $\A\rtimes_{\alpha} \G=\bigsqcup_{\gamma\in \G }B_\gamma$
where  \(B_\gamma:= A_{r(\gamma)}\) for \(\gamma\in \G\). The multiplication maps and involutions are defined by
  \(B_\eta\times B_\gamma \to B_{\eta\gamma}\),
  \((a,b) \mapsto a\cdot \alpha_\eta(b)\), and
  \(B_\gamma \to B_{\gamma^{-1}}\),
  \(a\mapsto \alpha_\gamma^{-1}(a^*)\). 
	\end{ex}
\subsection{Twisted groupoid $C^*$-dynamical systems vs twisted inverse semigroup actions}
In most of this paper (with a brief exception in Section 6) we will only consider Kumjian twists (\cite{Kumjian0}), i.e.\ twists by the trivial circle bundle $X\times \T$. 
Thus by a \emph{twist of a groupoid} $\G$ we mean  a topological groupoid $\Sigma$ such that we have  
a central groupoid
extension
$$
 X \times \T \stackrel{i}{\hookrightarrow}\Sigma \stackrel{\pi}{\onto} \G,
$$
where $ X\times \T$ is the group bundle over $X$. More specifically, cf.\ \cite[Definition 5.1.1]{Sims}, $i$ and $\pi$ are groupoid homomorphisms that restrict to the identity on 
the unit spaces $X=\Sigma^{(0)}\cong\G^{(0)}$, $i$ is injective, $\pi$ is surjective and $\pi^{-1}(X)=i(X\times \T)$. 
So suppressing  the map $i$ we may assume that $ X \times \T$ is a subgroupoid of $\Sigma$.
 This subgroupoid is assumed to be central in the sense that for every   
$\sigma\in \Sigma$ and $z\in \T$ we have $ (r(\sigma),z)\sigma=\sigma (s(\sigma),z)$; we denote  this common element  by
 $z \cdot \sigma$. Moreover, one also requires that $\Sigma$ is a locally trivial $\G$-bundle in the sense that every point $\gamma\in \G$ has a
bisection neighbourhood $U$ on which there exists a continuous section $c:U\to \Sigma$ satisfying 
$\pi\circ c= id_{U}$, and such that the map $U\times \T \ni (\gamma,z) \mapsto  z \cdot c(\gamma) \in \pi^{-1}(U)$
is a homeomorphism  $U\times \T\cong \pi^{-1}(U)$.
Thus $\Sigma$ is a principal $\T$-space and $\Sigma/\T\cong \G$. We will also call the pair $(\G,\Sigma)$ as above simply a \emph{twisted groupoid}.

Renault \cite{Renault}, \cite{Renault2} generalised Green's notion of twisted group actions \cite{Green}  to groupoid actions   as follows. 
We denote by $UM(A)$ the group of unitaries in the multiplier algebra of a $C^*$-algebra $A$.

\begin{defn}\label{def:twisted_groupoid_action}
A \emph{twisted groupoid $C^*$-dynamical system}, usually denoted by a quadruple   $(\A,\G,\Sigma,\alpha)$, 
consists of a twisted groupoid  $(\G,\Sigma)$ and a groupoid $C^*$-dynamical system  $(\A,\Sigma, \alpha)$ 
that are combined by a ``twisting'' continuous groupoid homomorphism $u: X \times  \T \to \bigsqcup_{x\in X} UM(A_x)$. This means that  
putting $u_x(z):=u(x,z)$, $x\in X$, $z\in \T$, we have a family $\{u_x\}_{x\in X}$ of 
group homomorphisms $u_x:\T\to UM(A_x)$, $x\in X$, such that 
\begin{enumerate}
\item\label{def:twisted_groupoid_action1} the map  $ \A  \times  \T \to \A$, where $ A_x \times  \T \ni ( a_x, z) \to a_xu_x(z)\in A_x$,
is continuous;
\item\label{def:twisted_groupoid_action2}
$\alpha_{(x,z)}:A_x\to A_x$  is  given by $u_x(z) (\cdot) u_x(z)^{-1}$, for every $(x,z)\in X \times \T$;

\item\label{def:twisted_groupoid_action3} $u_{r(\sigma)}(z)=\alpha_{\sigma} (u_{s(\sigma)}(z))$ 
for every $\sigma \in \Sigma$ and $z\in \T$; here 
$\alpha_{\sigma}:M(A_{s(\sigma)})\to M(A_{r(\sigma)})$ is the unique strictly continuous extension of $\alpha_{\sigma}:A_{s(\sigma)}\to A_{r(\sigma)}$.

\end{enumerate}
\end{defn}
	\begin{ex}\label{ex:bundle_from_twisted_action}
Every twisted groupoid  $C^*$-dynamical system $(\A,\G,\Sigma,\alpha)$  yields a saturated Fell bundle  over $\G$, cf.\ \cite[Section 3.1]{Lalonde}. 
We define it as the quotient $\A\rtimes_{\alpha} \Sigma/ \T$ of 
the Fell bundle $\A\rtimes_{\alpha} \Sigma$ associated to $(\A,\Sigma, \alpha)$ as in Example \ref{ex:bundle_from_action},
by the 
$\T$-action defined by 
\begin{equation}\label{eq:circle_action_twisted}
z \cdot(a,\sigma):= (a u_{r(\sigma)}(z)^{-1}, z\cdot \sigma ), \qquad (a,\sigma)\in \A\rtimes_{\alpha} \Sigma, z\in \T.
\end{equation}
We denote by $[a,\sigma]$ the corresponding class - the $\T$-orbit of $(a,\sigma)$. 
We also write  $\dot{\sigma}$ for the image of $\sigma\in \Sigma$ in $\G=\Sigma/\T$.
Then the fibers of the corresponding bundle are
$B_{\dot{\sigma}}=\{[a,\sigma]: a\in A_{r(\sigma)}\}$. 
The multiplication maps and involutions on $\A\rtimes_{\alpha} \Sigma$ factor through to 
well defined operations on the  quotient $\A\rtimes_{\alpha} \Sigma/ \T$. So 
the  formulas  
  $[a,\sigma]\cdot [b,\tau]:=[a\alpha_\sigma(b),  \sigma \tau]$, and
 $[a,\sigma]^{*}:=[\alpha_\sigma^{-1}(a^*),\sigma^{-1}]$, $a\in A_{r(\sigma)}$, $b\in A_{r(\tau)}$, 
yield the structure of the Fell bundle on $\B:=\bigsqcup_{\gamma\in \G} B_{\gamma}$. We will refer to it as 
the  \emph{Fell bundle associated to the twisted groupoid  $C^*$-dynamical system} $(\A,\G,\Sigma,\alpha)$.
 \end{ex}
\begin{rem}\label{rem:line_bundles_twists}
If $\A= X\times \C$ is a trivial line bundle (and $u_x$ are identities), then the construction in Example \ref{ex:bundle_from_twisted_action}
specializes to the standard construction of a continuous line bundle $\L=\bigsqcup_{\gamma\in \G} L_\gamma$ from a twisted groupoid $(\G,\Sigma)$, see \cite[2.5.iv]{Kumjian} or \cite{Re}.  
This gives a  part of the well known equivalence between continuous Fell line bundles over  $\G$ and twists of $\G$. 
For the other part, recall that if  $\L=\bigsqcup_{\gamma\in \G} L_\gamma$ is a continuous Fell line bundle then the circle bundle
$
\Sigma:=\{\xi \in \L: |\xi|=1\}
$ with the topology and multiplication inherited from $\L$ is a twist of $\G$.
Note that in this case both the line bundle $\L$ and the extension $\T\times X \hookrightarrow\Sigma \onto \G$ are  locally
trivial.
\end{rem}

Busby and Smith \cite{BusbySmith} introduced twisted actions by groups that involve cocycles rather than group extensions. As shown by Packer and Raeburn \cite{PackerRaeburn} this formalism covers
Green's approach to twisted crossed products.  Also the original Renault's twisted groupoid  $C^*$-algebras \cite{Renault0}  were defined in the spirit of Busby and Smith (dating back in fact already to Zeller-Meier \cite{Zeller-Meier}). 
Then Renault switched to  Green's approach in \cite{Renault}, \cite{Renault2},
as it was not clear whether the Packer-Raeburn result generalises to groupoids. 
This question is raised, for instance, in \cite[Remark 5.1.6]{Sims}, and  we were informed by Tristan Bice that the answer to it is negative.
Bice's example of a groupoid twist that does admit a continuous global section  comes from 
the Pedersen-Petersen $C^*$-Algebras, 
see  \cite{Bice}. The relevant $C^*$-algebra $\mathfrak{B}_1$ is a cross-sectional algebra of 
a continuous line bundle over  the principal groupoid $\G = \C P^1  \times\{0,1\}^2$
whose every continuous section  has to vanish by the Borsuk-Ulam Theorem. 

Nevertheless, we will show that Kumjian-Renault twisted actions are covered by 
a version of Busby-Smith twisted actions adapted to inverse semigroups by Buss and Exel  \cite{BussExel0}.
\begin{defn}\label{defn:twisted_inverse_semigroup_action}
A \emph{twisted action of an inverse semigroup}~\(S\)
  on a \(C^*\)-algebra~\(A\)
  consists of partial automorphisms
  \(\beta_t\colon D_{t^*}\to D_t\)
  of~\(A\)
  for \(t\in S\)
  -- that is, \(D_t\)
  is an ideal in~\(A\)
  and~\(\beta_t\)
  is a $*$-isomorphism -- and unitary multipliers
  \(\omega(t, u)\in UM(D_{t u})\)
  for \(t,u \in S\),
  such that \(D_1=A\)
  and the following conditions hold for \(r,t,u\in S\)
  and \(e,f\in E(S)\), where $E(S)$ is the set of units in $S$:
  \begin{enumerate}
  \item[(T1)] $D_{(rt)^*} = D_{t^*} \cap \beta_t^{-1}(D_{r^*})$ and  \(\beta_r\circ \beta_t= Ad_{\omega(r,t)}\beta_{r t}\);
  \item[(T2)]  \(\beta_r\big(a \omega(t,u)\big) \omega(r,t u)=
    \beta_r(a)\omega(r,t)\omega(r t,u)\) for \(a\in D_{r^*}\cap
    D_{t u}\);
  \item[(T3)]  \(\omega(e,f) = 1_{e f}\)
    and \(\omega(r,r^*r)=\omega(r r^*,r)=1_r\),
    where~\(1_r\) is the unit of \(M(D_r)\);
  \item[(T4)]  \(\omega(t^*,e)\omega(t^*e,t)a = \omega(t^*,t) a\)
    for all \(a\in D_{t^* e t}\).
  \end{enumerate}
\end{defn}

Given an \'etale groupoid $\G$, the set of its open bisections $\Bis(\G)$ forms naturally an inverse semigroup with operations
$$
U\cdot V:=\{\gamma\eta: \gamma\in U, \eta \in V, s(\gamma)=r(\eta)\},\qquad U^{*}=\{\gamma^{-1}:\gamma \in U\}, \qquad U,V\in \Bis(\G).
$$ 
An inverse semigroup $S\subseteq \Bis(\G)$ defines an inverse semigroup action $\{h_{U}\}_{U\in S}$ on the unit space $X=\G^{(0)}$,
where $h_U:=r\circ s|_{U}^{-1}:s(U)\to r(U)$. The transformation groupoid $X\rtimes_h S$ for this action is naturally isomorphic to $\G$ if and only if the semigroup $S$ is\emph{ wide}, that is
 \(\bigcup S = \G\) and \(U\cap V\) is a union of bisections
  in~\(S\) for all \(U,V\in S\); the corresponding notions and facts mentioned above are described for example in \cite[Section 2]{BartoszRalf2}.
  
Note that the conclusion of the next lemma does not depend on the choice of the section $c_U$; the resulting actions for different choices of sections need not explicitly coincide, but the resulting crossed products are isomorphic (see Theorem \ref{thm:twisted_action_groupoid_vs_inverse_semigroup} below).
	
\begin{lem}\label{lem:twisted_action_groupoid_vs_inverse_semigroup}
Let $(\A,\G,\Sigma,\alpha)$ be a twisted groupoid $C^*$-dynamical system with $\G$ \'etale. 
Let $S$ be the family of open bisections $U$  of $\G$ on which the $\T$-bundle $\pi:\Sigma\to \G$ is trivial, and for each  $U\in S$
choose a continuous section $c_U:U\to \Sigma$ that induces the homeomorphism $U\times \T\cong \pi^{-1}(U)$. 
If $U\subseteq X$ we choose $c_{U}$ to be given by  $c_U(x):=(x,1)\subseteq X\times \T\subseteq \Sigma$, for $x\in U$. Then the following statements hold.
\begin{enumerate}
\item\label{enu:inverse_semigroup_action1}   $S$ is a wide unital inverse subsemigroup of $\Bis(\G)$.
\item\label{enu:inverse_semigroup_action2}  For each $U\in S$ we may  treat $D_U:=C_0(\A|_{r(U)})$ as an ideal in $A:=C_0(\A)$. 
The formula
$$
\beta_{U}(a)(r(\gamma))= \alpha_{c_{U}(\gamma)}(a (s(\gamma)),\qquad  a\in D_{U^*}=C_0(\A|_{s(U)}),\,\, \gamma\in U,
$$
gives a well defined $*$-isomorphism $\beta_{U}:D_{U^*} \to D_{U}$.
\item\label{enu:inverse_semigroup_action4} For every $U$, $V\in S$ the formula
$$
\omega(U,V)(r(\gamma\eta)) = u( c_U(\gamma)c_V(\eta) c_{UV}(\gamma\eta)^{-1}), \qquad \gamma\in U, \eta\in V
$$ 
gives a well defined element in $UM(D_{UV})$.
\end{enumerate}
The quadruple $(A,S,\beta, \omega)$ described above yields a twisted action of an inverse semigroup in the sense 
of Definition \ref{defn:twisted_inverse_semigroup_action}.
\end{lem}
\begin{proof}
\eqref{enu:inverse_semigroup_action1} Let $U,V\in S$ and $c_U:U\to \Sigma$ and $c_V:V\to \Sigma$ be continuous sections that trivialise $\pi^{-1}(U)$ and $\pi^{-1}(V)$.
Then  putting $d_{U^{*}}(\gamma):=c_{U}(\gamma^{-1})^{-1}$, for $\gamma \in U^*$, and  $d_{UV}(\gamma\eta):=c_{U}(\gamma)c_{V}(\eta)$ for $\gamma\in U, \eta\in V$, 
we get continuous sections that trivializes $\pi^{-1}(U^*)$ and $\pi^{-1}(UV)$, respectively.
Thus $S$ is an inverse subsemigroup of $\Bis(\G)$. It is unital because $X\in S$ and it is wide as it is closed under inclusions and every $\gamma\in \G$ has 
a neighbourhood $U\in S$.  

\eqref{enu:inverse_semigroup_action2} Identifying sections of $\A|_{r(U)}$ with  sections of $\A$ that vanish outside the open set $r(U)$ 
we have that $D_U:=C_0(\A|_{r(U)})\subseteq C_0(\A)=A$ is an ideal. 
Using  continuity of the map $ \Sigma \, *_s\A  \ni(\sigma,a) \to \alpha_{\sigma}(a)\in \A$ and that 
$r:U\to r(U)$, $s:U\to s(U)$ and $c_U:U\to c_U(U)\subseteq \Sigma$ are homeomorphisms,
we see that the formula $\beta_{U}(f)(x):= \alpha_{c_{U}(r|_{U}^{-1}(x))}(f (s(r|_{U}^{-1}(x)))$, for $f\in C_0(\A|_{s(U)})$, $x\in r(U)$, gives a well defined element $\beta_{U}(f)$ in $C_0(\A|_{r(U)})$. This is exactly the 
formula in \eqref{enu:inverse_semigroup_action2}.
The map $\beta_{U}:D_{U^*} \to D_{U}$ is a $*$-isomorphism with the inverse given by 
 $\beta_{U}^{-1}(f)(s(\gamma))= \alpha_{c_{U}(\gamma)^{-1}}(f (r(\gamma)))$ for $f\in D_{U}=C_0(\A|_{r(U)})$, $\gamma \in U$.

\eqref{enu:inverse_semigroup_action4} Note first that we have  $UV\ni \gamma\eta\mapsto c_U(\gamma)c_V(\eta) c_{UV}(\gamma\eta)^{-1}\in r(UV)\times \T\subseteq X\times \T$ because 
$\pi(c_U(\gamma)c_V(\eta) c_{UV}(\gamma\eta)^{-1})=\pi(r(\gamma))$. Moreover this function is continuous, as so is the natural map $UV \ni \gamma \eta \to (\gamma, \eta) \in U \times V$.
Since $r:UV\to r(UV)$ is a homeomorphism and  $u: X \times  \T \to \bigsqcup_{x\in X} UM(A_x)$ is  continuous, we get that for every $a\in D_{UV}=C_0(\A|_{r(UV)})$ the formula 
$$
  [\omega(U,V)a]  (r(\gamma\eta))=u( c_U(\gamma)c_V(\eta) c_{UV}(\gamma\eta)^{-1}) a(r(\gamma\eta)) 
$$
defines an element of $C_0(\A|_{r(UV)})$. Now one readily infers that  $\omega(U,V)$ is a multiplier of $C_0(\A|_{r(UV)})$ with the adjoint
given by $\omega(U,V)^*=u( c_U(\gamma)c_V(\eta) c_{UV}(\gamma\eta)^{-1})^*$. Hence $\omega(U,V)\in UM(D_{UV})$. Note that in terms of the notation introduced in Definition \ref{def:twisted_groupoid_action} we have ($\gamma \in U, \eta \in V, \gamma \eta \in UV$)
\[ \omega(U,V) (r(\gamma \eta)) = u_{r(\gamma)} (p(c_U(\gamma) c_V(\eta) c_{UV}(\gamma\eta)^{-1} )),  \]
where $p:X\times \T\to \T$ is the projection onto the second coordinate.
Now we need to check conditions (T1)-(T4) in Definition \ref{defn:twisted_inverse_semigroup_action}. To simplify the notation in what follows we will sometimes write $\omega_{U,V}$ instead of $\omega(U,V)$. Choose  
$\gamma\in U\in S, \eta\in V\in S$ so that $\gamma\eta\in UV$. Then for $x:=r(\gamma\eta)=r(\gamma)$ and $a\in D_{(UV)^*}$, using 
Definition \ref{def:twisted_groupoid_action} \eqref{def:twisted_groupoid_action2}, we have
\begin{align*}
\left((Ad_{\omega_{U,V}} \circ \beta_{UV }) (a)\right) (x)&=\alpha_{c_U(\gamma)c_V(\eta) c_{UV}(\gamma\eta)^{-1}} (\alpha_{c_{UV}(\gamma\eta)}(a(x))
\\
&=\alpha_{c_U(\gamma)c_V(\eta)} (a(x))= \left((\beta_{U} \circ  \beta_{V}) (a)\right) (x).
\end{align*}
This proves (T1).  

Now take $a\in D_{Z^*}\cap D_{UV}=C_0(\A|_{s(Z)\cap r(UV)})$ where $U,V,Z\in S$. 
Take any $x\in r(ZUV)$ and choose $\gamma\in U$, $\eta\in V$ and $\sigma\in Z$ so that $x=r(\sigma\gamma\eta)$. 
Denoting by $p$ again the projection onto the second coordinate and using Definition \ref{def:twisted_groupoid_action}\eqref{def:twisted_groupoid_action3}
we get 
\begin{align*}
\beta_Z\Big(a\cdot  \omega_{U,V}\Big) (x)&=\alpha_{c_Z(\sigma)}\Big(a(s(\sigma)) u( c_U(\gamma)c_V(\eta) c_{UV}(\gamma\eta)^{-1}) \Big) \\
&=\alpha_{c_Z(\sigma)} \left(a(s(\sigma)) \cdot  u_{r(\gamma)}\Big( p(c_U(\gamma)c_V(\eta) c_{UV}(\gamma\eta)^{-1})\Big)\right)
\\
&=\beta_Z(a)(x) \cdot  u_{r(\sigma)}\Big( p(c_U(\gamma)c_V(\eta) c_{UV}(\gamma\eta)^{-1})\Big).
\end{align*}
Using that the extension is central, for the pointwise muliplication on  $r(ZUV)$ we have
\begin{align*}
p(c_Uc_V c_{UV}^{-1})p(c_Z c_{UV} c_{ZUV}^{-1})&=p(c_Z  p(c_Uc_V c_{UV}^{-1}) c_{UV} c_{ZUV}^{-1})=p(c_Z  c_Uc_V c_{UV}^{-1} c_{UV} c_{ZUV}^{-1})
\\
&=p(c_Z  c_Uc_V c_{ZUV}^{-1})=p(c_Z  c_U (c_{ZU}^{-1} c_{ZU})c_V c_{ZUV}^{-1})
\\
&=p(c_Z  c_U c_{ZU}^{-1}) p( c_{ZU}c_V c_{ZUV}^{-1}).
\end{align*}
Combining the above and using that $\omega_{Z,UV} (x)= u_{r(\sigma)}( p(c_Z(\sigma)c_{UV}(\gamma\eta) c_{ZUV}(\sigma\gamma\eta)^{-1})$,  we get

\begin{align*}
\beta_Z\Big(a\cdot  \omega_{U,V}\Big)\cdot \omega_{Z,UV}&=\beta_Z(a) \cdot  u\Big( p(c_U(\gamma)c_V(\eta)c_{UV}(\gamma\eta)^{-1}) c_Z(\sigma)c_{UV}(\gamma \eta) c_{ZUV}(\sigma\gamma \eta)^{-1}\Big) \\
&=\beta_Z(a) \cdot  u\big(c_Z(\sigma)c_U(\gamma) c_{ZU}(\sigma\gamma)^{-1}\big)  u \big(c_{ZU}(\sigma\gamma)c_{V}(\eta) c_{ZUV}(\sigma \gamma \eta)^{-1}\big)
\\
&=\beta_Z(a) \omega_{Z,U} \omega_{ZU,V}
\end{align*}
on $r(ZUV)$. This proves (T2).

Idempotents in $S$ correspond to open subsets of $X$. Thus axiom (T3) follows readily from our choice of $c_{U}$ for $U\subseteq X$.
To check (T4) let $U, V\in S$ where $V\subseteq X$, so that $c_V(V)=V\times \{1\}\subseteq V\times \T\subseteq \Sigma$. Then
on the set  $r(U^*VU)$ we have
$
\omega(U^*,V)\omega(U^*V,U) = u(c_{U^*}c_V c_{U^*V}^{-1}) u(c_{U^*V} c_{U}c_{U^*VU}^{-1})=u(c_{U^*}c_{U}c_{U^*VU}^{-1})=\omega(U^*,U) 
$, where we for example use the fact that $U^*VU$ is an idempotent, so that $c_{U^*VU} = c_{U^*U}|_{U^*VU}$.
\end{proof}
\emph{Saturated Fell bundles over inverse semigroups} (\cite{Exel:noncomm.cartan}, \cite{BussExel0}, \cite{BussExel}) are equivalent to inverse semigroup actions by Hilbert bimodules (\cite{BussMeyer}, \cite{BussExelMeyer}), whose  definition is relatively concise.
\begin{defn}[Definition 4.7, \cite{BussMeyer}]
Let $S$ be a unital inverse semigroup and $A$ a $C^*$-algebra. \emph{An action of $S$ on $A$ by Hilbert bimodules} is a semigroup 
$\mathcal{C}=\bigsqcup_{t\in S} C_t$ 
where each fiber $C_t$ is a Hilbert $A$-bimodule, $C_1=A$ is a trivial bimodule, and the semigroup product in $\mathcal{C}$ 
induces isomorphisms $C_t\otimes C_s\cong C_{ts}$ of Hilbert $A$-bimodules, for all $t,s\in S$. 
\end{defn}
\begin{rem} Any inverse semigroup action by Hilbert bimodules 
$\mathcal{C}=\bigsqcup_{t\in S} C_t$ induces uniquely defined canonical involutions
\(C_t^*\to C_{t^*}\), \(x\mapsto x^*\), and inclusion maps
\(j_{s,t}\colon C_t\to C_s\) for $t,s \in S$, \(t \le s\) that satisfy the
conditions required for a saturated Fell bundle,  as studied in \cite{Exel:noncomm.cartan}, \cite{BussExel0} and \cite{BussExel},
see 
\cite[Theorem~4.8]{BussMeyer}. Note further  that the structure of $\mathcal{C}=\bigsqcup_{t\in S} C_t$ is also uniquely determined by the multiplication
and involution in $\mathcal{C}$ (these operations determine the  Hilbert $A$-module operations on each fiber $C_t$).
\end{rem}
\begin{ex}\label{ex:bundle_from_semigroup_action}
Any twisted action  $(A,S,\beta, \omega)$ of an inverse semigroup $S$ on a $C^*$-algebra $A$ yields a natural action of $S$ on $A$ by Hilbert bimodules given by 
a family of the standard bimodules associated to partial automorphisms $\{\beta_t\}_{t\in S}$ of $A$. Namely, we put
$
\mathcal{C}:=\bigsqcup_{t\in S} D_t 
$, 
where each $D_t$ is a Hilbert $A$-bimodule with operations
$$
a\cdot \xi:=a\xi,\qquad \xi \cdot a:= \beta_t(\beta^{-1}_t(\xi)a),\qquad  \langle \xi,\psi\rangle_A:=\beta_t^{-1}(\xi^*\psi).
$$
where $a\in A$, $\xi,\psi\in D_t$. Multiplication between different fibers is given by 
$$
D_s\times D_t \ni (\xi_s, \xi_t)\longmapsto \beta_s(\beta^{-1}_s(\xi_s)\xi_t) \omega(s,t)\in D_{st}.
$$
The induced involution is given by 
$
D_s \ni \xi_s \longmapsto  \beta_s^{-1}(\xi_s^*) \omega(s^*,s)^*\in D_{s^*}, 
$
cf.\ \cite[page 250]{BussExel0}. In particular, 
every twisted groupoid $C^*$-dynamical system $(\A,\G,\Sigma,\alpha)$ yields 
a saturated Fell bundle $\bigsqcup_{U\in S}C_0(\A|_{r(U)})$ 
with operations coming from the twisted action  $(A,S,\beta, \omega)$ of an inverse semigroup $S$ on $A:=C_0(\A)$ 
as described in Lemma \ref{lem:twisted_action_groupoid_vs_inverse_semigroup}.  
\end{ex}
\subsection{The associated convolution $C^*$-algebras}
Let  $\bu=\bigsqcup_{\gamma\in \G }B_\gamma$ be a Fell bundle over an \'etale locally compact Hausdorff groupoid  $\G$.
The space  of compactly supported continuous sections $C_c(\bu)$ is naturally a $*$-algebra with operations 
\begin{equation}\label{eq:groupoid_algebra_operation}
  (f*g)(\gamma):= \sum_{r(\eta) = r(\gamma)}
  f(\eta)\cdot g(\eta^{-1}\cdot \gamma),\qquad
  (f^*)(\gamma) := f(\gamma^{-1})^*, 
\end{equation}
$f, g \in C_c(\bu)$, $\gamma \in \G$.
The
\emph{full section $C^*$-algebra} $C^*(\bu)$  is the maximal $C^*$-completion of
the $*$-algebra $C_c(\bu)$. Let $X:=\G^{(0)}$  be the unit space. The $C_0(X)$-$C^*$-algebra 
$
A:=C_0(\bu|_{X})
$
embeds naturally into $C^*(\bu)$, and the restriction map  $C_c(\bu)\ni f\to f|_X\in A$ extends to a conditional expectation
$C^*(\B)\to A$.  The \emph{reduced section $C^*$-algebra} $C_r^*(\bu)$ can be defined as the completion of the $*$-algebra $C_c(\bu)$ with respect to the minimal $C^*$-norm such that the restriction map $C_c(\bu)\to A$ is contractive. Thus the restriction 
extends to a faithful conditional expectation $E:C_r^*(\bu)\to A$. Note that if this conditional expectation is tracial, then  $A$ is commutative; but this condition is itself not sufficient. 
We will also need an explicit faithful representation of $C_r^*(\bu)$, as described for example in \cite{Kumjian}. 
Given $x \in \G^{(0)}$  put $\G_x:=\{\gamma \in \G:s(\gamma)=x\}$ and consider the Hilbert $B_{x}$-module
$V_x:=\bigoplus_{\gamma \in \G_x} B_\gamma$ (recall that the scalar product comes  from the bundle operations: if $a,b \in B_\gamma$, then $a^*b \in B_{\gamma^{-1}} B_\gamma\subseteq B_{\gamma^{-1} \gamma}= B_{s(\gamma)}$). Then we have a representation $\pi_x:C_r^*(\bu) \to B(V_x)$ determined by the formula
\begin{equation} 
\pi_x(f) (\bigoplus_{\gamma \in \G_x} \xi_\gamma ) = \bigoplus_{\gamma \in \G_x} \left(\sum_{\beta \in \G_x} f(\gamma \beta^{-1}) \xi_\beta \right),\label{pix}
\end{equation}
where  $f \in C_c(\bu)$, $\xi_\gamma \in B_\gamma$ for each $\gamma \in \G_x$, and the series $\sum_{\gamma \in \G_x} \langle \xi_\gamma, \xi_\gamma\rangle$ converges in $B_x$. The direct sum of the representations $\pi_x$ over $x \in \G^{(0)}$ is faithful.
\begin{rem}\label{rem:groupoid_identifications}
When equipped with the supremum norm, $C_0(\B)$ is a Banach space.
The embedding \(C_c(\B) \to C_0(\B)\)
  extends uniquely to an injective and contractive linear map
  \(C_r^*(\bu) \to  C_0(\B)\) which allows us to assume the identification
	$$
	C_r^*(\bu) \subseteq C_0(\B).
	$$
	Under this identification the $*$-algebraic operations in $C_r^*(\bu)$ are still given by \eqref{eq:groupoid_algebra_operation},
	see \cite[Proposition 7.10]{BartoszRalf2}. 
Moreover, for any $C^*$-completion $B=\overline{C_c(\B)}$ in a  $C^*$-norm $\|\cdot \|$  
that extends the supremum norm  on $C_c(\bu|_{\Gz})$  we have natural isometric embeddings $C_0(\B|_U)\hookrightarrow B$ for all open bisections 
$U\subseteq \G$. Indeed, if $f \in C_c(\B|_U)$, then 
$
\|f\|^2=\|f^**f\|_{C_0(\bu|_{\Gz})}=\max_{x\in \Gz} \|(f^**f)(x)\|=\max_{\gamma \in U} \|f^*(\gamma) f(\gamma)\|=\max_{\gamma \in U} \|f(\gamma)\|^2.
$
\end{rem}

\begin{ex}
Let $(\A,\G,\alpha)$ be a groupoid $C^*$-dynamical system and let $\A\rtimes_{\alpha} \G$
be the associated Fell bundle, cf.\ Example  \ref{ex:bundle_from_action}. 
The $*$-algebra $C_c(\A\rtimes_{\alpha} \G)$
consists of compactly supported continuous sections of $\A*_{r} \G=\bigsqcup_{\gamma\in \G }A_{r(\gamma)}$ with operations given by 
$$
  (f*g)(\gamma):= \sum_{r(\eta) = r(\gamma)}
  f(\eta)\cdot \alpha_{\eta}\big(g(\eta^{-1}\cdot \gamma)\big),\qquad
  (f^*)(\gamma) := \alpha_{\gamma}(f(\gamma^{-1}))^*, \;\;\; \gamma \in \G.
$$
The \emph{full} and \emph{reduced crossed product}, denoted by $C^*(\A,\G,\alpha)$ and $C^*_r(\A,\G,\alpha)$,
is by definition  $C^*(\A\rtimes_{\alpha} \G)$ and $C^*_r(\A\rtimes_{\alpha} \G)$, respectively.
\end{ex}

Let now $(\A,\G,\Sigma,\alpha)$  be a twisted groupoid  $C^*$-dynamical system  and let $\bu$
be the associated Fell bundle over $\G$, cf.\ Example  \ref{ex:bundle_from_twisted_action}. 
By definition, the \emph{full} and \emph{reduced crossed product}, denoted by $C^*(\A,\G,\Sigma, \alpha)$ and $C^*_r(\A,\G,\Sigma, \alpha)$,
is   $C^*(\bu)$ and $C^*_r(\bu)$, respectively. 
Using the fact  that  $\pi^{-1}(X)=X\times \T\subseteq \Sigma$ is a trivial bundle, we get 
that $\A\supseteq A_x\ni a\mapsto [a,(x,1)]\in B_x\subseteq \B|_X$ is the isomorphism
of $C^*$-bundles $\A\cong \B|_X$. Therefore we may  treat 
$$
A:=C_0(\A)\cong C_0(\bu|_{X})
$$
as a $C^*$-subalgebra of both $C^*(\A,\G,\Sigma, \alpha)$ and $C^*_r(\A,\G,\Sigma, \alpha)$.
In particular, we have a faithful conditional expectation $E:C^*_r(\A,\G,\Sigma, \alpha)\to A$.
In a similar manner we may identify the spaces of sections of $\B$ restricted to open bisections
on which the twist $\Sigma\onto \G$ is trivial.

\begin{lem}\label{lem:equivalent actions on A}
Let $(\A,\G,\Sigma,\alpha)$ be a twisted groupoid $C^*$-dynamical system.
 Let $\B$ be the associated Fell bundle over $\G$ (see Example \ref{ex:bundle_from_twisted_action}), 
and let $(A,S,\beta, \omega)$ be the associated twisted action   of an inverse semigroup $S$ on $A:=C_0(\A)$ 
described in Lemma \ref{lem:twisted_action_groupoid_vs_inverse_semigroup}.  
Put $B:=C^*_r(\A,\G,\Sigma, \alpha)$. 
We have the following  natural actions of the inverse semigroup $S$ by Hilbert bimodules on $A$, which are all naturally isomorphic:
\begin{enumerate}
\item\label{enu:iso_actions1} $\bigsqcup_{U\in S} B_U$  where $B_U:=\overline{C_c(\B|_U)} \subseteq B$ and all the operations are inherited from $B$ (the closure is with respect to the supremum norm).
\item\label{enu:iso_actions1.5} $\bigsqcup_{U\in S} C_0(\B|_U)$  with the  multiplication $
C_0(\B|_U)\times C_0(\B|_V)\ni (f,g) \mapsto  f*g \in C_0(\B|_{UV})
$ and the involution  $
C_0(\B|_U)\ni f \mapsto  f^* \in C_0(\B|_{U^*})
$ given by
$$
f * g \;(\gamma\eta):=f(\gamma)\cdot g(\eta), \qquad f^*(\gamma^{-1})=f(\gamma)^*, \quad \text{ for $\gamma\in U$ and $\eta \in V$}.
$$
\item\label{enu:iso_actions2}  $\bigsqcup_{U\in S} C_0(\A*_r\G|_U)$ with  the multiplication $
C_0(\A*_r\G|_U)\times C_0(\A*_r\G|_V)\ni (f,g) \mapsto  f*g \in C_0(\A*_r\G|_{UV})
$ and the involution  $
C_0(\A*_r\G|_U)\ni f \mapsto  f^* \in C_0(\A*_r\G|_{U^*})
$ given by
$$
f * g \;(\gamma\eta):=f(\gamma) \alpha_{c_U(\gamma)}(g(\eta)) u( c_U(\gamma)c_V(\eta) c_{UV}(\gamma\eta)^{-1}),
$$
$$
f^*(\gamma^{-1}):=\alpha_{c_U(\gamma)}^{-1}(f(\gamma)^*)u(c_{U^*}(\gamma^{-1})c_U(\gamma))^*,
$$
where $\gamma\in U$ and $\eta \in V$.
\item\label{enu:iso_actions3} $\bigsqcup_{U\in S}D_U$ associated to $(A,S,\beta, \omega)$  as in Example \ref{ex:bundle_from_semigroup_action},  so that the  
Hilbert $A$-bimodule structure on $D_U=C_0(\A|_{r(U)})$ is given by
$$
a\cdot f:=af,\quad f \cdot a:= \beta_U(\beta^{-1}_U(f)a),\quad  {_A\langle} f,g\rangle:=fg^*,\quad \langle f,g\rangle_A:=\beta_U^{-1}(f^*g),
$$
and the multiplication $D_U\times D_V \ni (f, g)\longmapsto \beta_U(\beta^{-1}_U(f) g) \omega(U,V)\in D_{UV}$.
\end{enumerate}
The  isomorphisms are given by 
$B_U\supseteq C_c(\B|_U)\ni f\stackrel{id}{\mapsto}f\in  C_0(\B|_U)$,  $C_0(\A|_{r(U)})\ni f\mapsto f\circ r\in  C_0(\A*_r\G|_U) 
$ and $
C_c(\A*_r\G|_U)\ni f \mapsto [f]  \in C_c(\B|_U)\subseteq B_U$ where $[f](\gamma):=[f(\gamma), c_U(\gamma)]$.
\end{lem}
\begin{proof}
For the isomorphism between \eqref{enu:iso_actions1} and \eqref{enu:iso_actions1.5}  recall that the identity map $B_U\supseteq C_c(\B|_U)\ni f\stackrel{id}{\mapsto}f\in  C_0(\B|_U)$ is isometric and hence extends to the
isomorphism $B_U\cong C_0(\B|_U)$, see Remark \ref{rem:groupoid_identifications}. That this isomorphism is consistent with the multiplication and the involution is obvious.

The isomorphism between the actions in \eqref{enu:iso_actions2} and \eqref{enu:iso_actions3}
 is straightforward, since for each $U\in S$ the homeomorphism $r:U\to r(U)$ induces an isomorphism  $\A*_r\G|_U\cong \A|_{r(U)}$ of bundles 
 (recall also that $c_U(x)=(x,1)$ for $U\subseteq X$, by definition).  We will show below that these actions are isomorphic to the action \eqref{enu:iso_actions1.5}.

Let $U\in S$. Recall that $U\times \T \ni (\gamma,z) \mapsto  z \cdot c_U(\gamma) \in \pi^{-1}(U)$
is a homeomorphism  $U\times \T\cong \pi^{-1}(U)$. 
This homeomorphism transfers $\A*_{r}\Sigma|_{\pi^{-1}(U)}$ into a bundle 
over $U\times \T$ whose quotient under the $\T$-action \eqref{eq:circle_action_twisted}
is isomorphic to $\A*_r\G|_U$. Thus $\B|_U=\A*_{r}\Sigma|_{\pi^{-1}(U)}/\T \cong \A*_r\G|_U$ 
and
the isomorphism between algebras presented in items (3) and (2) of the lemma  is given  by $
C_c(\A*_r\G|_U)\ni f \mapsto [f]  \in C_c(\B|_U)\subseteq B_U$ where $[f](\gamma):=[f(\gamma), c_U(\gamma)]$.
In particular, for any $f\in C_c(\A*_r\G|_U)$, $g\in C_c(\A*_r\G|_V)$, $\gamma\in U$ and $\eta\in V$we have
\begin{align*}
[f][g](\gamma\eta)&=[f](\gamma)[g](\eta)=[f(\gamma), c_U(\gamma)] [g(\eta), c_V(\eta)]
=[f(\gamma) \alpha_{c_U(\gamma)}(g(\eta)), c_U(\gamma)c_V(\eta)]
\\
&=[f(\gamma) \alpha_{c_U(\gamma)}(g(\eta)) u(c_{UV}(\gamma\eta) c_V(\eta)^{-1} c_U(\gamma)^{-1})^{-1}, c_{UV}(\gamma\eta) ]
\\
&=[f(\gamma) \alpha_{c_U(\gamma)}(g(\eta)) u( c_U(\gamma)c_V(\eta) c_{UV}(\gamma\eta)^{-1}), c_{UV}(\gamma\eta) ]
\\
&=[f * g (\gamma\eta), c_{UV}(\gamma\eta) ].
\end{align*}
Thus we have a semigroup isomorphism $\bigsqcup_{U\in S}  C_c(\A*_r\G|_U) \to C_c(\B|_U)$. 
It preserves the involution as we have 
\begin{align*}
[f]^*(\gamma^{-1})&=[f](\gamma)^*=[f(\gamma), c_{U}(\gamma)]^*=[\alpha_{c_{U}(\gamma)^{-1}}(f(\gamma)^*), c_{U}(\gamma)^{-1}]
\\
&=[\alpha_{c_{U}(\gamma)}^{-1}(f(\gamma)^*), \big(c_{U^*}(\gamma^{-1})c_U(\gamma)\big)^{-1}c_{U^*}(\gamma^{-1})]
\\
&=[\alpha_{c_{U}(\gamma)}^{-1}(f(\gamma)^*) u(c_{U^*}(\gamma^{-1})c_U(\gamma))^{-1},c_{U^*}(\gamma^{-1})]
\\
&=[f^*](\gamma^{-1}).
\end{align*}

\end{proof}

\begin{rem}
The last lemma could be equally well formulated with  $B=C^*(\A,\G,\Sigma, \alpha)$ or more generally with $B$ being	
any quotient of  $C^*(\A,\G,\Sigma, \alpha)$ by an ideal that intersects $A$ trivially, cf.\ Remark \ref{rem:groupoid_identifications}.
\end{rem}

If  $\mathcal{C}=\bigsqcup_{t\in S} C_t$ is an action of an inverse semigroup 
$S$ by Hilbert bimodules on $A=C_1$, then  $\bigoplus_{t\in S} C_t$ with operations
from $\mathcal{C}$ is naturally a $*$-algebra and using the induced inclusion maps one can 
define a quotient $*$-algebra of $\bigoplus_{t\in S} C_t$ whose full completion
and a completion allowing a faithul weak conditional expectation are respectively the 
full and reduced $C^*$-algebras of $\mathcal{C}$, see   \cite{BussExelMeyer}. 
They coincide with
the $C^*$-algebras associated to $\mathcal{C}$ treated as a Fell bundle over $S$ as defined in 
\cite{Exel:noncomm.cartan}, \cite{BussExel0} or \cite{BussExel}. \label{page:twisted_crossed_products}
The  \emph{full $A\rtimes_{\beta}^{\omega} S$} and \emph{reduced crossed product $A\rtimes_{\beta, r}^{\omega} S$ of an inverse semigroup twisted action} $(A,S,\beta, \omega)$ are by definition the full and reduced $C^*$-algebra of the associated Fell bundle over $S$, 
described in Example  \ref{ex:bundle_from_semigroup_action},  see  \cite{BussExel0}.

As explained below Remark \ref{rem:line_bundles_twists}, the following result can be viewed as a generalisation
of  \cite[Proposition 5.1]{PackerRaeburn} from group actions to groupoid actions.
It also generalises a part of \cite[Theorem 7.2]{BussExel0} that characterises $C^*$-algebras of groupoid twists as crossed products by twisted actions.
\begin{thm}
\label{thm:twisted_action_groupoid_vs_inverse_semigroup}
Let $(\A,\G,\Sigma,\alpha)$ be a twisted groupoid $C^*$-dynamical system with $\G$ \'etale and let  $(A,S,\beta, \omega)$ 
be the twisted inverse semigroup action defined in Lemma \ref{lem:twisted_action_groupoid_vs_inverse_semigroup}, so that $S$ denotes the family of open bisections $U$  of $\G$ on which the $\T$-bundle $\pi:\Sigma\to \G$ is trivial. Let $S_0\subseteq S$ be any  inverse semigroup which is wide in $\G$, as defined before Lemma \ref{lem:twisted_action_groupoid_vs_inverse_semigroup} (one may take $S_0=S)$. Then the corresponding twisted crossed products  are  canonically isomorphic:
$$
C^*(\A,\G,\Sigma, \alpha)\cong A\rtimes_{\beta}^{\omega} S_0,\qquad C^*_r(\A,\G,\Sigma, \alpha)\cong A\rtimes_{\beta,r}^{\omega} S_0.
$$
\end{thm}
\begin{proof}
Let $\mathcal{C}=\bigsqcup_{U\in S} C_0(\B|_U)$ be the inverse semigroup action by Hilbert bimodules on $A$ described in Lemma \ref{lem:equivalent actions on A}\eqref{enu:iso_actions1.5}. This is a standard 
inverse semigroup action associated to the Fell bundle $\B$ over $\G$, see  \cite[Example 2.9]{BussExel} or 
\cite[section 7.1]{BartoszRalf2}. Thus restricting $\mathcal{C}$ to any inverse subsemigroup $S_0\subseteq S$    which is wide in $\G$
we get $C^*(\bigsqcup_{U\in S_0} C_0(\B|_U))\cong C^*(\B)$ and $C^*_{r}(\bigsqcup_{U\in S_0} C_0(\B|_U))\cong C^*_r(\B)$ by 
\cite[Theorem 2.13]{BussExel}, or \cite[Corollary 5.6]{BussMeyer}, and \cite[Theorem 8.11]{BussExelMeyer},  cf.\ also
\cite[Propositions 7.6, 7.9]{BartoszRalf2}.
This gives the assertion as $\B$ is isomorphic to the bundle associated to $(\A,\G,\Sigma, \alpha)$ by Lemma \ref{lem:equivalent actions on A}.
\end{proof}

 \section{Multipliers for  twisted groupoid dynamical systems} \label{sect:Fell bundles}

In this section we discuss the notion of mapping multipliers for Fell bundles over groupoids,
and we characterise a natural class of completely positive multipliers for Fell bundles coming from twisted groupoid dynamical systems.  We 
consider the maps acting on the reduced $C^*$-algebras, as opposed to the universal ones, studied in detail for example in \cite{RamsayWalter}, \cite{RenaultFourier} or \cite{RenaultAFA} (albeit only in the case of trivial line bundles).

\begin{defn} Let $\B$ be a Fell bundle over the groupoid $\G$. A \emph{mapping $\B$-multiplier} is a function $G\ni \gamma\mapsto h(\gamma)\in B(B_\gamma)$ which is continuous in the sense  that
the formula 
\begin{equation}\label{eq:multiplier_definition}
(m_h f )(\gamma)= h(\gamma)f(\gamma)\;\;\;\; f\in  C_c(\bu), \gamma \in \G, 
\end{equation}
defines a map (a priori with no continuity requirement) $m_h:C_c(\bu)\to C_c(\bu)$. We say that a  mapping $\B$-multiplier $h$ is \emph{bounded} (resp.\emph{ completely bounded} or \emph{completely positive}) 
if the map $m_h$ extends to to a bounded (resp. completely bounded, completely positive) map on $C_r^*(\bu)$.
\end{defn}
\begin{ex}[Scalar multipliers] Every continuous function $h:\G\to\C$ can be treated as a mapping $\B$-multiplier 
for every Fell bundle $\B$ over $\G$, as then \eqref{eq:multiplier_definition} defines $m_h:C_c(\bu)\to C_c(\bu)$. We call such multipliers \emph{scalar $\B$-multipliers}. Recall that a function $h:\G \to \C$ is called \emph{positive-definite} if for any $x \in \Gz$ and a finite set $F\subseteq \G_x$ the matrix 
$$
\left[h(\gamma\eta^{-1})\right]_{\eta,\gamma \in F}
$$
is positive-definite. Takeishi showed in \cite[Lemma 4.2]{Takeishi} that every compactly supported positive-definite function is a completely positive scalar $\bu$-multiplier for any $\G$-bundle $\bu$ and remarked that the compact support condition can be dropped; we record the corresponding fact, together with its converse, in Proposition \ref{prop:positive-defnite_vs_completely_positive}.
\end{ex}
\begin{ex}[$M(\B|_{\Gz})$-valued multipliers] Let again $\B$ be a Fell bundle  over $\G$. Then $\B|_{\Gz}$ is a continuous field of $C^*$-algebras and so we have a pullback bundle $\B|_{\Gz}*_r\G=\bigsqcup_{\gamma\in \G} B_{r(\gamma)}$. Every fiber $B_\gamma$ is naturally a left $B_{r(\gamma)}$-Hilbert module, and hence also a left 
$M(B_{r(\gamma)})$-Hilbert module. Hence every  \emph{strictly continuous} section $h$ of the multiplier bundle $M(\B|_{\Gz}*_r\G):=\bigsqcup_{\gamma\in \G} M(B_{r(\gamma)})$ can be treated as a mapping $\B$-multiplier. 
\emph{Strict continuity} here means that 
 for any continuous section $a$ of $\B|_{\Gz}*_r\G$ the pointwise product $ha$ is also  continuous as a section of $\B|_{\Gz}*_r\G$.
If $\B$ is saturated then such multipliers coincide with mapping $\B$-multipliers with the values in adjointable operators on the fibers $B_\gamma$, $\gamma\in \G$. Sometimes we will call mapping multipliers of this type simply $\B$-multipliers.

Recall that if $\B$ is a Fell bundle associated to  a
twisted groupoid  $C^*$-dynamical system $(\A,\G,\Sigma,\alpha)$, then $\B|_{\Gz}\cong\A$ with the isomorphism given by $\A\ni a_\gamma \mapsto [a_\gamma,1]$. So
every strictly continuous section $h$ of $M(\A*_r\G)=\bigsqcup_{\gamma\in \G} M(A_{r(\gamma)})$ is a $\B$-multiplier;  we have
$(m_h  f)(\dot{\sigma})=[h(\dot{\sigma})a,\sigma]$ for each $\sigma \in \Sigma$, $f\in C_c(\B)$ and $a \in A_{r(\sigma)}$ such that $f(\dot{\sigma}) = [a, \sigma]$.
\end{ex}
In general it is not clear how to describe explicitly completely positive mapping $\B$-multipliers. But we can do it
for $\B$-multipliers on Fell bundles coming from twisted groupoid  $C^*$-dynamical systems.

\begin{defn}\label{defn:positive_definite} Let  $(\A,\G,\Sigma,\alpha)$ be a twisted groupoid  $C^*$-dynamical system.
We say that a section $h$ of $M(\A*_r\G)$ is \emph{positive-definite} if for any $x \in \Gz$ and any finite set $F
\subseteq \G_x$, for some section $c:F\to \Sigma$   the matrix 
$$
\left[\alpha_{c(\gamma)}^{-1} \left(h(\gamma\eta^{-1})\right)\right]_{\eta,\gamma \in F} \in M_{|F|}(M(A_x))
$$
is positive.
\end{defn}

We will now show that the definition above in fact does not depend on the choice of the section, and discuss the central $\B$-multipliers.

\begin{lem}\label{lem:independence_of_section}
Let  $(\A,\G,\Sigma,\alpha)$ be a twisted groupoid  $C^*$-dynamical system and  let $h$ be a section of $M(\A*_r\G)$.
\begin{enumerate}
\item\label{enu:independence_of_section1} Let  $x \in \Gz$ and let $F
\subseteq \G_x$ be a finite set.
If there is  a  section $c:F\to \Sigma$ such that  the matrix 
$
\left[\alpha_{c(\gamma)}^{-1} \left(h(\gamma\eta^{-1})\right)\right]_{\eta,\gamma \in F}
$ is positive, then the corresponding matrix is positive for any section $c':F\to \Sigma$. 
\item \label{enu:independence_of_section2} 
The map $m_h$ is
a $C_c(\A)$-bimodule map if and only if  
$
h(\gamma)\in ZM(A_{r(\gamma)})$, for all $\gamma \in \G.
$

\end{enumerate}
\end{lem}
\begin{proof}\eqref{enu:independence_of_section1}
Consider two sections $c,c':F \to \Sigma$ for $x, F$ as above. Then there is a function $f:F\to \T$ such that $c=fc'$. Note that using Definition \ref{def:twisted_groupoid_action}\eqref{def:twisted_groupoid_action2} we obtain for any element $z \in \mathbb{T}$ and $\sigma \in \Sigma$ the following equality: $\alpha_{z\sigma} (\cdot)= u_{r(\sigma)}(z)\alpha_\sigma(\cdot)
u_{r(\sigma)}(z)^{-1}$, and further by Definition \ref{def:twisted_groupoid_action}\eqref{def:twisted_groupoid_action3} we obtain $\alpha_{z\sigma} (\cdot)= \alpha_\sigma(u_{s(\sigma)}(z) \cdot u_{s(\sigma)}(z)^{-1})
$. This is in turn equivalent to $\alpha_{z\sigma}^{-1} (\cdot)= u_{s(\sigma)}(z)^{-1}\alpha_\sigma^{-1}(\cdot ) u_{s(\sigma)}(z) 
$.
Thus for any $\{a_\gamma\}_{\gamma\in F}\subseteq M(A_x)$ putting $b_\gamma:=a_{\gamma}u_x(f(\gamma))$, $\gamma \in F$ we get 
$$
\sum_{\gamma,\eta \in F}  a_{\gamma} \alpha_{c'(\gamma)}^{-1} \left(h(\gamma\eta^{-1})\right) a_{\eta}^*
= \sum_{\gamma,\eta \in F}  b_{\gamma} \alpha_{c(\gamma)}^{-1} \left(h(\gamma\eta^{-1})\right) b_{\eta} ^*. 
$$
This proves that  positivity of $[\alpha_{c(\gamma)}^{-1} (h(\gamma\eta^{-1}))]_{\eta,\gamma \in F}$ implies that of $
[\alpha_{c'(\gamma)}^{-1} (h(\gamma\eta^{-1}))]_{\eta,\gamma \in F}.
$

\eqref{enu:independence_of_section2}
The formula \eqref{eq:multiplier_definition} implies that each $m_h$  is a right $C_c(\A)$-module map. Hence it is also  a $C_c(\A)$-bimodule map if  
$
h(\gamma)\in ZM(A_{r(\gamma)})$, for all $\gamma \in G$. Conversely, assume $m_h$ is a left $A$-module map. 
 For arbitrary  $a,b\in A_{r(\sigma)}$, $\sigma \in \Sigma$,   
there are  functions $f\in C_c(\A)$ and  $g\in C_c(\B)$ supported on open bisections and such that
$f(r(\sigma))=a$, $g(\dot{\sigma})=[b,\sigma]$. Then, denoting the canonical image of $f$ in $ C_c(\B)$ by $\dot{f}$ (so that $\dot{f}(\dot{\sigma}) = [f(\sigma), r(\sigma)]$), we have
$$
[h(\dot{\sigma})ab,\sigma]=m_h(\dot{f}\cdot g)(\dot{\sigma})=(\dot{f}\cdot m_h(g))(\dot{\sigma})=[a,r(\sigma)][h(\dot{\sigma})b,\sigma]=[ah(\dot{\sigma})b, \sigma].
$$
 Hence $h(\dot{\sigma})ab=ah(\dot{\sigma})b$, for every $a,b\in A_{r(\sigma)}$. This implies $h(\dot{\sigma})\in ZM(A_{r(\sigma)})$.
 \end{proof}

Following the classical argument from the context of Herz-Schur multipliers (see for example \cite[Subsection 5.6]{BO}) 
we obtain the next result.
\begin{prop}\label{prop:pdCrossed}
Let  $\B$ be a  Fell bundle associated to a twisted groupoid  $C^*$-dynamical system
$(\A,\G,\Sigma,\alpha)$, and let  $h$  be  a section of
 $M(\A*_r\G)$. The following conditions are equivalent:
\begin{enumerate}
\item\label{it:pdCrossed1} $h$ is a completely positive  $\B$-multiplier, i.e \eqref{eq:multiplier_definition}  defines
a completely positive map $m_h:C_r^*(\bu)\to C_r^*(\bu)$;
\item\label{it:pdCrossed2} $h$ is strictly continuous, bounded and  positive-definite.
\end{enumerate}
If these equivalent conditions hold then 
$h$ is  central (i.e. $
h(\gamma)\in ZM(A_{r(\gamma)})$, for all $\gamma \in \G$),  $m_h$ is a $C_0(\A)$-bimodule map, and
$
\|m_h\|=\sup_{\gamma\in \G} \|h(\gamma)\|=\sup_{x\in \Gz}\|h(x)\|.
$
	\end{prop}

\begin{proof}
\eqref{it:pdCrossed1}$\Rightarrow$\eqref{it:pdCrossed2} 
Clearly,  $h$ is a mapping $\B$-multiplier, i.e. it defines $m_h:C_c(\bu)\to C_c(\bu)$, if and only if $h$ is strictly continuous. Since $m_h:C_r^*(\bu)\to C_r^*(\bu)$ is completely positive, it is automatically bounded and $\|m_h\|=\sup_{\lambda}\|m_h(\mu_\lambda)\|$ for an approximate unit $(\mu_\lambda)$ in $A:=C_0(\A)$. This readily implies
that $\|m_h\|=\sup_{x\in \Gz}\|h(x)\|\leq \sup_{\gamma\in \G} \|h(\gamma)\|$. Moreover, for any $\gamma\in \G$ and 
$\varepsilon>0$ we may find an open bisection $U$ containing $\gamma$ and  a section $a\in C_c(\B|_{U})\subseteq C_c(\B)\subseteq C^*_r(\B)$ such that
 $\|h(\gamma)\|\leq \|h(\gamma)a(\gamma)\|+\varepsilon$ and $\|a\|=1$ (cf.\ the second part of Remark \ref{rem:groupoid_identifications}). Hence 
$\|h(\gamma)\|\leq \|m_h(a)(\gamma)\|+\varepsilon\leq \|m_h\|+\varepsilon$.
This implies that $\sup_{\gamma\in \G} \|h(\gamma)\| =\|m_h\|$, and so $h$ is bounded.
Furthermore, since  $m_h:C_r^*(\bu)\to C_r^*(\bu)$ is (completely) positive it is $*$-preserving, 
and by its very definition \eqref{eq:multiplier_definition}, $m_h$ is a right $C_0(\A)$-module map. Hence $m_h$ is a  $C_0(\A)$-bimodule map:
$
m_h(ab)=m_h(b^*a^*)^*=a m_h(b)$, $a\in C_c(\A)$, $b\in C_c(\B)$. Thus $h$ is central by Lemma \ref{lem:independence_of_section} \eqref{enu:independence_of_section2}.  

To see that $h$ is positive-definite,  fix $x\in \G^{(0)}$ 
and a finite set $F\subseteq  \G_x$. For every $\gamma \in F$ let $f_\gamma\in C_c(\bu)$ be a section
 supported on an open  bisection. The matrix 
\(
(f_\gamma f_\eta^*)_{\gamma, \eta \in F}\in M_{|F|}(C^*_r(\bu))
\)
 is positive (recall that the representation $\pi_x$ and the module $V_X$ were introduced in the formula \eqref{pix} and the discussion preceding it). So also the operator
\[ 
B:=\pi_x^{(|F|)} \circ m_h^{(|F|)}  ((f_\gamma f_\eta^*)_{\gamma, \eta \in F}) \in B(V_x^{\oplus |F|})   
\]
is positive.  Let us choose a section $c:F\to \Sigma$  and write $f_\gamma(\gamma)=[a_{\gamma},c(\gamma)] $ for  \(a_\gamma \in A_{r(\gamma)}\)
and $\gamma \in F$. Note that  
$$
(f_\gamma f_\eta^*) (\gamma \eta^{-1}) =
f_\gamma(\gamma) f_\eta(\eta)^* = [a_{\gamma},c(\gamma)]  [\alpha_{c(\eta)^{-1}}(a_{\eta}^*),c(\eta)^{-1}]
=[a_\gamma\alpha_{c(\gamma)c(\eta)^{-1}} (a_\eta^*),c(\gamma)c(\eta)^{-1}].
$$
Let then $\xi=\bigoplus_{\eta\in F} \xi_{\eta}= \bigoplus_{\eta\in F} \bigoplus_{\beta \in \G_x} \xi_\eta(\beta)\in V_x^{\oplus |F|}$ be such that 
$ \xi_\eta(\beta)\in B_{\beta}$ is zero unless $\beta=\eta$. Write $ \xi_\eta(\eta)=[b_{\eta},c(\eta)]$ for 
\(b_\eta \in A_{r(\eta)}\) and $\eta \in F$. We have
\begin{align*} 
0\leq\langle \xi, B \xi \rangle &
= \sum_{\gamma, \eta \in F} \langle \xi_\gamma , \pi_x(m_h (f_\gamma f_\eta^*)) \xi_\eta \rangle 
= \sum_{\gamma, \eta \in F} \langle \xi_\gamma(\gamma) , m_h (f_\gamma f_\eta^*)(\gamma\eta^{-1})\xi_\eta(\eta)\rangle
\\&
=\sum_{\gamma, \eta \in F}  \left\langle [b_{\gamma},c(\gamma)], [h(\gamma \eta^{-1}) a_\gamma\alpha_{c(\gamma)c(\eta)^{-1}} (a_\eta^*),c(\gamma)c(\eta)^{-1}] \cdot [b_{\eta},c(\eta)]\right\rangle
\\&
=\sum_{\gamma, \eta \in F}   \alpha_{c(\gamma)}^{-1} \left(b_\gamma^* h(\gamma \eta^{-1}) a_\gamma\alpha_{c(\gamma)c(\eta)^{-1}} (a_\eta^*) \alpha_{c(\gamma)c(\eta)^{-1}}(b_\eta)\right)
\\& 
=\sum_{\gamma, \eta \in F}  \alpha_{c(\gamma)}^{-1} (b_\gamma^*) \alpha_{c(\gamma)}^{-1} (h(\gamma \eta^{-1})) 
\alpha_{c(\gamma)}^{-1} (a_\gamma) \alpha_{c(\eta)^{-1}}  (a_\eta^*b_\eta).
\end{align*}
It now remains to note that we can choose the functions $f^i_\gamma$ so that the corresponding $a^i_\gamma$  approximate $1_{M(A_{r(\gamma)})}$ in the strict topology. Then taking the limit over $i$ we obtain
\[
0\leq \sum_{\gamma, \eta \in F}  \alpha_{c(\gamma)}^{-1} (b_\gamma^*) \alpha_{c(\gamma)}^{-1} (h(\gamma \eta^{-1})) \alpha_{c(\eta)}^{-1} (b_\eta).\]
Now as the choice of elements $b_\gamma$ was arbitrary, the matrix 
$ \left[\alpha_{c(\gamma)}^{-1} \left(h(\gamma\eta^{-1})\right)\right]_{\gamma, \eta \in F} $ is positive in 
$M_{|F|}(M(A_x))$.

\eqref{it:pdCrossed2}$\Rightarrow$\eqref{it:pdCrossed1} Since $h$ is bounded we may normalize it so that it is a strictly continuous section of $M(\A*_r\G)$  bounded by $1$. Fix $x \in \G^{(0)}$ 
and choose a  section $c:\G_x\to \Sigma_x\subseteq \Sigma$. 
By Lemma  \ref{lem:independence_of_section}, the formula 
$k_x(\gamma, \eta)=\alpha_{c(\gamma)^{-1}} h(\gamma\eta^{-1})$, $\gamma, \eta \in \G_x$ defines a positive-definite 
 kernel $k_x:\G_x \times \G_x \to M(A_x)$. Hence by \cite[Theorem 2.3]{Mur} there exists a Hilbert $A_x$-module $\Hil_x$ and a map $\zeta: \G_x \to  \L(A_x,\Hil_x)$ (with values in the unit ball) such that for each $\kappa, \beta \in \G_x$ we have  
\begin{equation}\label{eq:Murphy_given}
\alpha_{c(\kappa)}^{-1} (h(\kappa\beta^{-1})) = \zeta (\kappa)^* \zeta (\beta).
\end{equation}
Note that we can view   $\L(A_x,\Hil_x)$ as a Hilbert module $\Hilt_x$ over $M(A_x)\cong \L(A_x)$, cf.\ \cite[Proposition 1.2]{RaeburnThompson}.

The isomorphisms   $A_{r(\gamma)}\ni a_\gamma \mapsto  [a_{\gamma},c(\gamma)]\in B_\gamma$, extended to direct sums, allow us to identify the right $A_x$-module $V_x = \bigoplus_{\gamma \in \G_x} B_\gamma $ (described in the construction of the regular representation)  with the $A_x$-module $\bigoplus_{\gamma \in \G_x} A_{r(\gamma)}$ with operations   given by 
$$
\left(\bigoplus_{\gamma \in \G_x} a_{\gamma}\right )\cdot a=\bigoplus_{\gamma \in \G_x} a_{\gamma}\alpha_{c(\gamma)}(a), \qquad \left\langle \bigoplus_{\gamma \in \G_x} b_{\gamma}, \bigoplus_{\gamma \in \G_x} d_{\gamma}\right\rangle_{A_x}=\sum_{\gamma \in \G_x} \alpha_{c(\gamma)}^{-1}(b_{\gamma}^*d_{\gamma}).
$$
We  equip $V_x$ with a natural left action  $\phi_x:M(A_x)\to \L(V_x)$, given by the formula:
\[
 \phi_x (m) \cdot \bigoplus_{\gamma \in \G_x} a_{\gamma}:= \bigoplus_{\gamma \in \G_x} \alpha_{c(\gamma)}(m)a_{\gamma},\qquad m \in M(A_x),
\bigoplus_{\gamma \in \G_x} a_{\gamma}\in V_x.\]
 Thus we can form a new $A_x$-Hilbert module $W_x=\Hilt_x \otimes_{M(A_x)} V_x$. Consider the map $\Theta_x: V_x \to W_x$ given by 
\[
\Theta_x \left( \bigoplus_{\gamma \in \G_x} b_\gamma \right) = \bigoplus_{\gamma \in \G_x} \zeta(\gamma) \otimes b_\gamma. 
\]
One readily verifies that $\Theta_x$ is a  contraction and we can define a completely positive contractive map $\Psi_x: B(V_x) \to B(V_x)$ by the formula 
\[ 
\Psi_x (T) = \Theta_x^*(1_{\Hilt_x} \otimes T) \Theta_x, \;\;\; T \in B(V_x).
\]
The formula \eqref{pix} allows us to check that we have
\begin{equation}\label{eq:multiplier_over_pi} 
\pi_x (m_h f) = \Psi_x(\pi_x(f)), \;\;\; f \in C_c(\bu). 
\end{equation}
Indeed, by linearity and continuity, to establish the above equality it suffices to verify that  for each 
$\eta, \gamma \in \G_x$ and $b_\gamma \in B_\gamma, d_\eta \in B_\eta$ we have
$\langle b_\gamma, \pi_x (m_h f) d_\eta \rangle_{A_x} = \langle  b_\gamma, \Psi_x(\pi_x(f)) d_\eta \rangle_{A_x}$.  
Using \eqref{eq:Murphy_given} we get $(m_hf) (\gamma \eta^{-1})d_{\eta}=\phi_x(\langle \zeta(\gamma), \zeta(\eta) \rangle_{M(A_x)})f (\gamma \eta^{-1})d_{\eta}$. Therefore, using the definition of the internal tensor product of Hilbert modules and of the representation $\pi_x$ we obtain 
\begin{align*} 
\langle b_\gamma, (m_hf) (\gamma \eta^{-1}) d_\eta \rangle_{A_x} &=  \langle b_\gamma, \phi_x(\langle \zeta(\gamma), \zeta(\eta) \rangle_{M(A_x)}) \cdot ( f(\gamma \eta^{-1}) \cdot d_\eta) \rangle_{A_x}
\\&=
\langle \zeta(\gamma) \otimes b_\gamma, \zeta(\eta) \otimes \pi_x(f) d_\eta \rangle_{A_x} 
\\&=  
\langle \Theta_x (b_\gamma), (I \otimes \pi_x(f)) \Theta_x (d_\eta) \rangle_{A_x} 
\\&=  
 \langle  b_\gamma, \Psi_x(\pi_x(f)) d_\eta \rangle_{A_x}. 
\end{align*}
Now, as the direct sum of the representations $\pi_x$ is faithful on $C_r^*(\bu)$, \eqref{eq:multiplier_over_pi}  implies that the map $m_h$ is contractive and completely positive.
\end{proof}

A very similar proof gives also the following result: note that we allow below arbitrary $\G$-bundles, but consider only scalar multipliers.

\begin{prop}\label{prop:positive-defnite_vs_completely_positive}
	Let $h:\G\to \C$ be a bounded continuous function and let $\bu$ be a $\G$-bundle. Consider the following conditions:
	\begin{itemize}
		\item[(i)] $h$ is positive-definite;
		\item[(ii)] $h$ is a completely positive scalar $\bu$-multiplier.
	\end{itemize}	
	Then (i) implies (ii), and if $\bu$ is ``strongly saturated'' in the sense that 
	 for every $\gamma \in \G$ there is an approximate unit for $B_{s(\gamma)}$
contained in $\langle B_{\gamma}, B_{\gamma}\rangle_{B_{s(\gamma)}}$,
	then (ii) implies (i).
	Thus (i)$\Leftrightarrow$(ii)  for continuous line bundles and for  bundles coming from twisted groupoid actions.
\end{prop}

\begin{rem}
 If we consider a transformation groupoid $\G:= \Gamma \times X$ (where a discrete group $\Gamma$ acts on a compact space $X$), and consider the trivial line bundle over $\G$, then a function $h \in C(\G)$ can be canonically identified with a map	
$\tilde{h}:\Gamma \to C(X)$. The positive-definiteness condition for $h$ reads then as follows: 
for any $x \in X$ and a finite set $F\subseteq \Gamma$ we have that the matrix
\[\left[\tilde{h} (f^{-1} g) (f\cdot)\right]_{f,g \in F} \in M_{|F|}(C(X)) \]
is positive. This is the condition appearing for example in \cite[Equation (3.2)]{dong_ruan} (up to a change of convention) or in \cite[Definition 4.2]{BedosConti2};
see also a discussion in \cite[Section 2]{mstt}.
\end{rem}

\section{The Haagerup trick for 
twisted groupoid dynamical systems} \label{Sect:HTrick} 

Throughout the section we fix a   twisted groupoid  $C^*$-dynamical system $(\A,\G,\Sigma,\alpha)$ where 
$\G$ is an \'etale locally compact Hausdorff groupoid. We let \(\B\) be the Fell bundle associated to it and  consider the following $C^*$-inclusion
$$
 A:= C_0(\A) \subseteq B=C^*_r(\A,\G,\Sigma, \alpha) 
$$
equipped with the canonical faithful conditional expectation $E:B \to A$. 
 In Proposition \ref{prop:pdCrossed} we have seen that
strictly continuous positive-definite sections of $M(\A*_r\G)$ give rise to completely positive maps on $B$.
In this section we prove the converse. To this end,  we develop  a generalisation 
of `the  Haagerup Trick', cf.\ \cite[Lemma 2.5]{Haagerup}. To avoid technicalities, and also because
the approximation properties we want to study in sequel sections seem to be designed for the unital context, from now on we will assume here that
\begin{quote}
	the bundle $\A$ has a \emph{continuous  unital section}, i.e. 
each fiber $A_x$, $x\in \Gz$ is unital and the unit section $\Gz \ni x\mapsto 1_x\in A_x\subseteq \A$ is continuous.
\end{quote}
Here and below $1_x$ denotes the unit in the algebra $A_x$, $x \in \Gz$.
Our standing  assumption implies that $A$ is unital
if and only if $\Gz$ is compact. It also implies that $M(\A*_r\G)=\A*_r\G$ and for sections of this bundle strict continuity  is just continuity. Further  we may assume that $C_0(\G)\subseteq Z(A)$ by identifying functions in $C_0(\G)$ with their product by the unit section.

In the following proposition we use the notation introduced in Lemma \ref{lem:twisted_action_groupoid_vs_inverse_semigroup}, 
(see also Lemma \ref{lem:equivalent actions on A}), we denote by $\B$ the Fell bundle associated to  $(\A,\G,\Sigma,\alpha)$ in  Example  \ref{ex:bundle_from_twisted_action},
and use the identification $C_r^*(\B)\subseteq C_0(\B)$ from Remark \ref{rem:groupoid_identifications}.
\begin{prop}[Haagerup Trick]\label{prop:Haagerup Trick}
	Suppose $\Phi:B\to B$ is a bounded right $A$-module map. 
	Fix an element $a\in A$. 
	For each $\gamma\in \G$ choose a  bisection $U\in S$ containing $\gamma$ and a  section   $f\in C_c(\A|_{r(U)})$ 
	which attains units at a neighborhood of $r(\gamma)$, and define
	$$
	f_\gamma(\eta):=[f(r(\eta)), c_U(\eta)],\qquad \eta \in U.
	$$
	Then $f_\gamma\in C_c(\B|_{U})\subseteq B$ and we define
	\begin{equation}\label{eq:Haagerup_multiplier}
	h^\Phi(\gamma):=\Phi(f_\gamma)(\gamma) f_\gamma(\gamma)^*=E( \Phi(f_\gamma) f_\gamma^*)(r(\gamma))\in A_{r(\gamma)}.
	\end{equation}
	The formula above  does not depend on the choice of $U$, $f$ and $c_U$, and defines a bounded continuous section $h^\Phi$ of  
	$\A*_r\G$ (we have $\|h^\Phi\|_{\infty}\leq \|\Phi\|$). Moreover, if $\Phi$ is completely positive then $h^\Phi$ is positive-definite. 
	
	\end{prop}

\begin{proof} The element $f_\gamma$ corresponds to the section $f$ under the isomorphism 
$C_0(\B|_{U})\cong C_0(\A*_r\G|_{U})\cong C_0(\A|_{r(U)})$ from  Lemma \ref{lem:equivalent actions on A}. In particular, $f_\gamma\in C_c(\B|_{U})\subseteq B$. To see that the definition of 
$h^\Phi(\gamma)$ does not depend on the choice of the section $c_U:U\to \Sigma$,
	let $c_U':U\to \Sigma$  be another continuous section that trivializes the $\T$-bundle, and put 
	$f_{\gamma}'(\eta):=[f(r(\eta)), c_U'(\eta)]$,
	$\eta \in \G$.  Then $c(r(\eta)):=c_U'(\eta)c_U(\eta)^{-1}$, $\eta\in U$, defines a continuous section of  $\Sigma|_{r(U)}\cong r(U)\times \T \onto r(U)$. Thus identifying $c(x)$ with the uniquely determined element of $\mathbb{T}$ we see that $s(U)\ni s(\eta) \mapsto u_{s(\eta)}( c(r(\eta)))\in A_x\subseteq \A$ is also a continuous section. Let $h\in C_0(s(U))$ be any function which is equal to one on $s(\supp(f_\gamma))$. Then $u(s(\eta)):=h(s(\eta))u_{s(\eta)}( c(r(\eta)))$, $\eta\in U$, defines an element of  $C_0(\A|_{s(U)})\subseteq A$.
	For any $\eta\in \supp(f_\gamma)\subseteq  U$ we have 
	\begin{align*}
	f_{\gamma}'(\eta)&=[f(r(\eta)), c_U'(\eta)]=[f(r(\eta)), c(r(\eta)) c_U(\eta)]=[f(r(\eta))u_{r(\eta)}(c(r(\eta))),  c_U(\eta)]
	\\
	&=[f(r(\eta))\alpha_{c_U(\eta)}\big(u_{s(\eta)}(c(r(\eta)))\big),  c_U(\eta)]
	=[f(r(\eta))\alpha_{c_U(\eta)}\big(u(s(\eta))\big),  c_U(\eta)]
	\\
	&=
	(f_{\gamma} u)(\eta).
	\end{align*}
	Hence $f_{\gamma}'= f_{\gamma} u$ where $u\in A$ is such that $uu^*$ is the unit section on $s(\supp(f_\gamma))$. As $\Phi$ is a right $A$-module map and $(uu^*)(x)=1_x$ for all 
	$x\in s(\supp(f_\gamma))$, we obtain
	$$
	\Phi(f_{\gamma}') (f_{\gamma}')^*=\Phi(f_{\gamma}) uu^* f_\gamma^* 
	=	\Phi(f_{\gamma}) f_{\gamma}^*.
	$$
	Independence of $h^\Phi(\gamma)$  from the choice of $U$ will follow once we show its independence from the choice of $f$.
	To show the latter, let $f'\in C_c(\A|_{r(U')})$, $U'\in S$, be another section  which  is unital at a neighborhood of $r(\gamma)$, and put 
	$(f_\gamma')(\eta):=[f'(r(\eta)), c_U(\eta)]$. Then $f_\gamma|_V=f_\gamma'|_V$ for some  neighborhood $V$ of $\gamma$ (find $V$ such that both $f$ and $f'$ are unital
	on $r(V)\subseteq r(U)\cap r(U')$). Let $h\in C_c(s(V))$ be such that $h(s(\gamma))=1$. 
	We may treat $h$ as an element of $A$ by multiplying it by the unit section of $\A$.  Then $f_\gamma h=f_\gamma' h$ and using that $\Phi$ is a right $A$-module map we get
	\begin{align*}
	\Phi(f_{\gamma})(\gamma) (f_{\gamma})^* (\gamma)&=\Phi(f_\gamma)(\gamma) h(s(\gamma)) \overline{h(s(\gamma))}(f_{\gamma}^* (\gamma))
	=\Phi(f_{\gamma}h)(\gamma) (f_{\gamma}h)^* (\gamma)
	\\
	&=\Phi(f_{\gamma}'h)(\gamma) (f_{\gamma}'h)^* (\gamma) =\Phi(f_\gamma')(\gamma) h(s(\gamma)) \overline{h(s(\gamma))}(f_{\gamma}'^* (\gamma))
		\\
	&=\Phi(f_{\gamma}')(\gamma) (f_{\gamma}')^* (\gamma).
		\end{align*}
		Hence the formula \eqref{eq:Haagerup_multiplier} gives a well defined section $h^\Phi$ of  
	$\A*_r\G$; the fact that we also have the expression involving the conditional expectation follows as $f_\gamma$ is supported on a bisection.  For any $\gamma \in \G$ 
 the section $f_\gamma$ is  given by $f$ with $f(r(\eta))=1_{r(\eta)}$ for all \(\eta\in V\) for   some open neighbourhood $V$ of $\gamma$.
Hence by the independence on the choice of a right $f$ established above we have $h^\Phi(\eta)= E( \Phi(f_\gamma) f_\gamma^*) (r(\eta))$ for any $\eta \in V$. 
	As $E( \Phi(f_\gamma) f_\gamma^*)\in A=C_0(\A)$ is continuous, this implies that $h^\Phi$ is continuous at $\gamma$.  
  Also we may choose $f$ so that $\|f\|=1$, equivalently $\|f_\gamma\|=1$, and this implies that 
	$\|h^\Phi(\gamma)\|=\|E( \Phi(f_\gamma) f_\gamma^*)(r(\gamma))\|\leq \|\Phi\|$. 
	Thus the section $h^\Phi$ is bounded by $\|\Phi\|$.

	Now suppose that $\Phi$ is completely positive. 
	To see that $h^\Phi$ is positive-definite fix $x\in \Gz$ and a finite set $F \subseteq \G_x$.
	For each $\gamma \in F$ choose $f_\gamma$ as in 
	\eqref{eq:Haagerup_multiplier} supported on $U_\gamma\in S$. Define a section $c:F\to \Sigma$ by $c(\gamma):=c_{U_\gamma}(\gamma)$, $\gamma \in F$. Note that for any $\eta, \gamma \in F$, the product $f_\gamma * f_\eta^*$ is  supported on an open bisection in $S$ and it is the unit section on some neighbourhood of 
	$\gamma\eta^{-1}$.
	Hence 
	\begin{align*}
	h^\Phi(\gamma\eta^{-1})&= \Phi(f_\gamma * f_\eta^*)(\gamma\eta^{-1})(f_\gamma* f_\eta^*)(\gamma\eta^{-1})^* =
	\Phi(f_\gamma * f_\eta^*)(\gamma\eta^{-1})f_\eta(\eta) f_\gamma(\gamma)^*.
	\end{align*}
	A simple  computation using the formulas from Example \ref{ex:bundle_from_twisted_action} and the fact that $f_\gamma(\gamma)=[1_{A_{r(\gamma)}}, c_U(\gamma)]$ shows that 	$f_\gamma(\gamma)^*  h^\Phi(\gamma\eta^{-1}) f_\gamma(\gamma)= \alpha_{c(\gamma)}^{-1}\big(h^\Phi(\gamma\eta^{-1})\big)$
	and $f_\gamma(\gamma)^*f_\gamma(\gamma)=1_{s(\gamma)}$. Thus we get
	$$
	\alpha_{c(\gamma)}^{-1}\big(h^\Phi(\gamma\eta^{-1})\big)= f_\gamma(\gamma)^* \Phi(f_\gamma * f_\eta^*)(\gamma\eta^{-1})f_\eta(\eta)=
	E\left(f_\gamma^* *\Phi(f_\gamma * f_\eta^*) * f_\eta\right)(x),
	$$
	where in the last equality we use again the fact that both $f_\gamma$ and $f_\eta$ are supported on bisections.
	Thus the matrix
	$$
	\left[\alpha_{c(\gamma)}^{-1}\big(h^\Phi(\gamma\eta^{-1})\big)\right]_{\eta,\gamma \in F}
	=\left[E\left(f_\gamma^* *\Phi(f_\gamma * f_\eta^*) * f_\eta\right)(x)\right]_{\eta,\gamma \in F}
	$$
	is positive  as $\Phi$  and $E$ are  completely positive.  
\end{proof}
\begin{rem} In the case of a groupoid dynamical system $(\A,\G,\alpha)$  (without a twist) we may replace $\B$ by $\A*_r\G$ and then the map \eqref{eq:Haagerup_multiplier} 
	induced by $\Phi:B\to B$
	could be defined by 
	$h^\Phi(\gamma)=\Phi(f)(\gamma)
	$  where $f\in C_c(\A*_r\G|_{U})$
	is  supported on an open bisection \(U\)  and such that  $f(\eta)=1_{r(\eta)}$ on some neighbourhood of $\gamma$. We also invite the reader to check that when the crossed product in question is just the usual group $C^*$-algebra (i.e.\ $G$ is a discrete group and the bundle $\A$ is trivial) the formula \eqref{eq:Haagerup_multiplier} reduces to the one from the classical Haagerup trick.
\end{rem}
\begin{rem} Given any bounded mapping $\B$-multiplier $h\in C_b(\A*_r\G)$ the map  $m_h:B\to B$ is a bounded right $A$-module map and  
$h^{m_h}=h$.
 But  for a general bounded right $A$-module $\Phi:B\to B$  it may happen that $m_{h^{\Phi}}\neq \Phi$.
\end{rem}

We now wish to characterise  when a section obtained by the Haagerup Trick is of $C_0$ or $C_c$ class.
To this end, we use the right Hilbert $A$-module $B_E$ associated to the canonical expectation $E:B\to A$. 
Thus
\begin{equation}\label{eq:def_B_E}
B_E:=\overline{B}^{\|\cdot\|_{E}}, \quad \langle x, y\rangle_{E}:=E(x^*y), \qquad x,y \in B\subseteq B_E.
\end{equation}
As is common, we will write $\mathcal{L}(B_E)$ for the algebra of all adjointable operators on $B_E$, for $\xi, \eta \in B_E$ the symbol $\Theta_{\xi, \eta} \in \mathcal{L}(B_E)$ will denote the obvious `rank-one operator' and $\K(B_E)$ will denote the closed linear span of rank-one operators inside $\mathcal{L}(B_E)$ (see for example \cite{Lance}).
Denoting by $\|\cdot\|_{E}$ the norm in $B_E$, by $\|\cdot\|_{\infty}$ the supremum norm on $C_0(\bu)$ and 
by $\|\cdot \|_r$ the norm in $B=C^*_r(\B)$ we have $\|f\|_{\infty}\leq  \|f\|_{E}\leq \|f\|_{r}$ for all
 $f\in C_c(\bu)$. In particular, we may assume the identifications
$$
C_c(\bu)\subseteq B\subseteq B_E=\{f\in C_0(\bu): \sup_{x\in \Gz} \|\sum_{\gamma\in \G_x} f^{*}(\gamma) f(\gamma) \|_{A_x}<\infty\}\subseteq C_0(\bu),
$$
cf.\ Remark \ref{rem:groupoid_identifications}.

In general, for an element  $h\in C_b(\A*_r \G)$ the formula \eqref{eq:multiplier_definition}, 
does not define a bounded map $m_h$ on $B=C^*_r(\bu)$, 
and  a bounded
right $A$-module map $\Phi:B\to B$ does not need to extend to a map  $\tilde{\Phi}\in \mathcal{L}(B_E)$. This can be seen already if one considers $\A$ to be trivial and $\G$ to be a group, so that we are in the context of Herz-Schur multipliers.
However, we have the following result.

\begin{prop}\label{prop:Hilbert multipliers_fell_groupoids}
	The space  $C_b(\A*_r\G)$ of continuous bounded sections  equipped with pointwise operations
	and supremum norm is a $C^*$-algebra.
	We have an injective $*$-homomorphism $C_b(\A*_r\G)\ni h\to \widetilde{m}_{h}\in \mathcal{L}(B_E)$ 
	where
	\begin{equation}\label{eq:multiplier_on_Hilbert_module}
	(\widetilde{m}_{h} (a))(\gamma):= h(\gamma)\cdot a(\gamma), \qquad a\in  C_c(\bu)\subseteq B_E, \,\, \gamma \in \G.
	\end{equation}
	Moreover, if $h\in C_b(\A*_r\G)$ then 
	\begin{enumerate}
		\item\label{enu:multipliers_fell_groupoids1}  \(h\in C_c(\A*_r\G)\) if and only if 
		$\widetilde{m}_h\in \mathcal{F}(C_c(\bu)):=\text{span}\{\Theta_{f,g}: f,g\in C_c(\bu)\}$.
	
		\item\label{enu:multipliers_fell_groupoids2} \(h\in C_0(\A*_r\G)\) if and only if \(\widetilde{m}_h\in \mathcal{K}(B_E)\).
			\end{enumerate}
Similarly, if $\Phi:B\to B$ 	 is a bounded right $A$-module map that extends to a bounded map $\tilde{\Phi}:B_E\to B_E$, then we have
\begin{enumerate}
	\setcounter{enumi}{2}
	\item\label{enu:multipliers_fell_groupoids3}  \(h^\Phi\in C_c(\A*_r\G)\) if  $\tilde{\Phi}\in \mathcal{F}(C_c(\bu))$.
	\item\label{enu:multipliers_fell_groupoids4}
 \(h^\Phi\in C_0(\A*_r\G)\) if  $\tilde{\Phi}\in \mathcal{K}(B_E)$.
\end{enumerate}		
\end{prop}
\begin{proof}
	The first claim is immediate. 
	Let $h\in C_b(\A*_r\G)$. Clearly,  \eqref{eq:multiplier_on_Hilbert_module} defines a map 
	$\widetilde{m}_{h}:  C_c(\bu)\to C_c(\bu)$. Let  $a,b \in C_c(\bu)$ and $x\in \Gz$. Since  
	$ E(a^* b)(x)=(a^* * b)(x) = \sum_{\gamma\in \G_x} a(\gamma)^*b(\gamma)
	$, we get
	\begin{equation}\label{multadjoint}
	E(\widetilde{m}_{h}(a)^*b)(x)= \sum_{\gamma\in \G_x} a(\gamma)^* h(\gamma)^{*}b(\gamma)
	=E(a^* \widetilde{m}_{h^*}(b))(x).
	\end{equation}
	That is, $\langle \widetilde{m}_{h}(a), b\rangle_{E}=\langle a, \widetilde{m}_{h^*}(b)\rangle_{E}$. 
	Similarly,
	\begin{align*}
	E(\widetilde{m}_{h}(a)^* \widetilde{m}_{h}(a))(x)&=  \sum_{\gamma\in \G_x}  a(\gamma)^*|h(\gamma)|^2 a(\gamma)
	\leq \|h\|^2 \sum_{\gamma\in \G_x}  a(\gamma)^*a(\gamma)=  \|h\|^2    E(a^* a)(x).
	\end{align*}
	This implies that $\langle \widetilde{m}_{h}(a), \widetilde{m}_{h}(a)\rangle_{E}\leq \|h\| \langle a,a\rangle_{E}$ in $A$, and hence 
	$
	\|\widetilde{m}_{h}(a)\|^2_{E}\leq \|h\|^2 \|a\|^2_{E}
	$.  So $\widetilde{m}_{h}$ is bounded, and $\|\widetilde{m}_{h}\|\leq \|h\|$.
	Accordingly,   $\widetilde{m}_{h}:C_c(\bu)\to C_c(\bu)$ extends to a bounded map on 
	$B_E$. By \eqref{multadjoint} this map is adjointable, with the  adjoint equal to $\widetilde{m}_{h^*}$.
	Thus the map $C_b(\A*_r\G)\ni h\to \widetilde{m}_{h}\in \mathcal{L}(B_E)$  is well defined and $*$-preserving.
	By  \eqref{eq:multiplier_on_Hilbert_module} it is immediate that this map is also linear, multiplicative and injective.
	
	\eqref{enu:multipliers_fell_groupoids1} and  	\eqref{enu:multipliers_fell_groupoids3}:
	To  show 
	that $h\in C_c(\A*_r\G)$ implies  $\widetilde{m}_h \in \mathcal{F}(C_c(\bu))$  it suffices to consider the case when  $h$ is supported on an open bisection
	\(U\in S\) where $S$ is as in Lemma \ref{lem:twisted_action_groupoid_vs_inverse_semigroup}, cf.\ also Lemma \ref{lem:equivalent actions on A}.  
		Using that $C_0(\A|_{r(U)})$ is a $C^*$-algebra and $r$ induces isomorphisms  $\bu|_{U}\cong \A|_{r(U)}$ and  $C_c(\bu|_{U})\cong C_c(\A|_{r(U)})$, we can find elements $f,g\in C_c(\A*_r\G|_{U})$
	such that $h(\gamma)=f(\gamma)g(\gamma)^*\in A_{r(\gamma)}$ for every $\gamma \in \G$. 
	Then for any  $a\in C_c(\A*_{r} \G|_{V})$, $V\in S$,  and $\gamma \in U\cap V$ we get
	\begin{align*}
	h(\gamma)[a](\gamma)&=  [f(\gamma)g(\gamma)^*a(\gamma), c_U(\gamma)]=[f(\gamma),c_U(\gamma)][\alpha_{c_U(\gamma)}^{-1}
	\big(g(\gamma)^*a(\gamma)\big), s(\gamma)]
	\\	&
	= [f](\gamma)\langle [g], [a]\rangle_A (s(\gamma))=( \Theta_{[f],[g]} [a])(\gamma).
	\end{align*}
	Hence $\widetilde{m}_{h}= \Theta_{[f],[g]}\in \mathcal{F}(C_c(\G))$. 
	
	Suppose now that $h\in C_b(\A*_r\G)$ is any section such that $\widetilde{m}_{h}\in \mathcal{F}(C_c(\bu))$  
or that $h=h^\Phi$ where $\tilde{\Phi}\in \mathcal{F}(C_c(\bu))$. Then either  $\widetilde{m}_{h}$ or $\tilde{\Phi}$
	is of the form $T=\sum_{i=1}^n \Theta_{f_i,g_i}$ for  $f_i\in C_c(\bu|_{U_i})$, $g_i\in C_c(\bu|_{V_i})$, 
	where $U_i, V_i\subseteq \G$ are  open bisections,  $i=1,...,n$.  For each $\gamma \in \G$ choose  $f_\gamma$ is as in Proposition \ref{prop:Haagerup Trick}. Then 	
	\begin{align*}
	(Tf_{\gamma})(\gamma)&=\left(\sum_{i=1}^n f_i \cdot E (g_i^*\cdot  f_{\gamma})\right)(\gamma)=
	\sum_{i=1}^n f_i(\gamma)g_i(\gamma)^*f_{\gamma}(\gamma)=\left(\sum_{i=1}^n f_i(\gamma)g_i(\gamma)^*\right) f_{\gamma}(\gamma).
	\end{align*}
		If $T=\widetilde{m}_{h}$ we get $h(\gamma) f_{\gamma}(\gamma)=\left(\sum_{i=1}^n f_i(\gamma)g_i(\gamma)^*\right)  f_{\gamma}(\gamma)$.
		And if $T=\tilde{\Phi}$ we get $h(\gamma)=h^\Phi(\gamma)= \left(\sum_{i=1}^n f_i(\gamma)g_i(\gamma)^*\right)$.
		In both cases, we conclude that $h(\gamma)\neq 0$ if and only if $\sum_{i=1}^n f_i(\gamma)g_i(\gamma)^* \neq 0$, so the support of $h$ is compact, that is $h\in C_c(\A*_r\G)$.

	\eqref{enu:multipliers_fell_groupoids2} and  	\eqref{enu:multipliers_fell_groupoids4}:
	By what we have proved  $C_c(\A*_r\G)\ni h\to \widetilde{m}_{h}\in \mathcal{F}(C_c(\B))\subseteq \mathcal{K}(B_E)$ 
	is a contractive linear  map. Hence it extends uniquely to a contractive map 
	$C_0(\A*_r\G)\ni h\to \widetilde{m}_{h}\in \mathcal{K}(B_E)$, and it is easy to check that this extension is compatible with the formula defining the map in Proposition \ref{prop:Hilbert multipliers_fell_groupoids}.
	Thus if  $h\in C_0(\A*_r\G)$, then  $\widetilde{m}_{h}\in \mathcal{K}(B_E)$.
	
	Conversely, assume that $h\in C_b(\A*_r\G)$ is such that $\widetilde{m}_{h}\in \mathcal{K}(B_E)$ 
	or that  $h=h^\Phi$ where $\tilde{\Phi}\in \mathcal{K}(B_E)$.  Let $T\in \mathcal{K}(B_E)$ stand either for 
	$\widetilde{m}_{h}$ or for $\tilde{\Phi}$ and let $\varepsilon >0$. Since 
	$\overline{\mathcal{F}(C_c(\bu))}=\mathcal{K}(B_E)$, there is 
	$R=\sum_{i=1}^n \Theta_{f_i,g_i}$ with $f_i\in C_c(\bu|_{U_i})$, $g_i\in C_c(\bu|_{V_i})$, where $U_i, V_i\in S$,  $i=1,...,n$,
	and $\|T-R\|_{\L(B_E)}<\varepsilon/2$.
	It suffices to show  that the closed set $\{\gamma\in \G: \|h(\gamma)\|\geq \varepsilon\}$ is contained in the following compact subset of $\G$: 
	$$
	K:=\bigcup_{i=1}^n (s|_{V_i})^{-1} ( s(\supp f_i)).
	$$
	To see this, note that for any $\gamma\notin K$ we may find $U$ and  $f_\gamma\in C_c(\bu|_U)$  as in Proposition \ref{prop:Haagerup Trick} 
	such that  $U\subseteq \G\setminus K$ and $\|f_\gamma\|_{\infty}=1$.  Then $\|f_\gamma\|_E^2=\|E(f_\gamma^*f_\gamma)\|_{\infty}=\|f_\gamma^*f_\gamma\|_{\infty}=1$. Moreover, $R(f_\gamma)=0$, as for any $\eta\in \G$, 
	$$
	(R (f_\gamma))(\eta)
	=\sum_{i=1}^n f_i(\eta)g_i(s|_{V_i}^{-1}(s(\eta)))^* f_\gamma(s|_{V_i}^{-1}(s(\eta)))=0.
	$$ 
	If $T=\tilde{\Phi}$ then
	$ 
	\|h^\Phi(\gamma)\|=\|\Phi(f_\gamma)(\gamma) f_\gamma(\gamma)^*\|\leq \|\Phi(f_\gamma)(\gamma)\| 
	\leq \|\tilde{\Phi}(f_{\gamma})\|_E=\|\tilde{\Phi}(f_{\gamma}) -R(f_\gamma)\|_{E}<\varepsilon.
	$ 
	If $T=\widetilde{m}_{h}$, then
	$ 
	\|h(\gamma)\|=\|h(\gamma)f_{\gamma}(\gamma)\| \leq \|\widetilde{m}_{h}(f_{\gamma})\|_{E}
	=\|\widetilde{m}_{h}(f_{\gamma}) -R(f_{\gamma})\|_{E}<\varepsilon.
	$ 
	Thus in both cases we get  that $\{\gamma\in \G: \|h^\Phi(\gamma)\|\geq \varepsilon\}\subseteq K$.
	Hence $h\in C_0(\B)$.	
	\end{proof}

The uniform convergence on compact subsets of $\G$ can be rephrased in terms of 
pointwise convergence in the $\|\cdot\|_E$ norm.

\begin{lem}\label{lem:convergence_multipliers}
A net  $(h_i)_{i \in \Ind}\subseteq	  C_b(\A*_r\G) $ converges to the unit section  uniformly on compact sets,
i.e. $\sup_{\gamma\in K}\|h_i(\gamma)-1_{r(\gamma)}\|\to 0$ for every compact $K\subseteq \G$, if 
and only if  $(\widetilde{m}_{h_i})_{i \in \Ind}$ converges  pointwise on $C_c(\bu)\subseteq B_E$ to the identity operator
(so that if $(\widetilde{m}_{h_i})_{i \in \Ind}$ is bounded it converges  pointwise  to the identity operator on $B_E$).
Also	if $(\Phi_i)_{i\in I}$ is a net of  bounded  $A$-bimodule maps on  $C_r^*(\bu)$  which converges to identity pointwise on $C_c(\bu)\subseteq C_r^*(\bu)$ in the  $\|\cdot\|_E$ norm, then $(h^{\Phi_i})_{i \in \Ind}$ converges to the unit section uniformly on compact subsets.
\end{lem}
\begin{proof}
	
	Suppose that  the net $(h_i)_{i \in \Ind} $ as above converges to the unit section uniformly on compact sets.  To show 
	that 
	$\lim_{i \in \Ind}\|\widetilde{m}_{h_i}(b) - b\|_{E} = 0$ for  $b\in C_c(\bu)$  we may assume that $b\in C_c(\bu|_U)$ where $U\in S$ is an open bisection. Thus there is   $a\in C_c(\A*_r\G|_{U})$   such that $b(\gamma)=[a(\gamma), c_U(\gamma)]$ for all $\gamma \in \G$. Since
	$K:=\overline{\{\gamma: b(\gamma)\neq 0\}}$ is compact we then get
	\begin{align*}
	\|\widetilde{m}_{h_i}(b) -b \|_{E}^2&=\sup_{x\in \Gz}\| \left((\widetilde{m}_{h_i}(b) -b)^*(\widetilde{m}_{h_i}(b) -b)\right)(x)\|
		\\
	&=\sup_{\gamma \in K} \| a(\gamma)^*(h_i(\gamma) -1_{r(\gamma)})^*(h_i(\gamma) -1_{r(\gamma)})a(\gamma)\|
	\\
	&\leq \sup_{\gamma \in K} \| a(\gamma)^*a(\gamma)\| \|h_i(\gamma) -1_{r(\gamma)}\|^2
	\\
	&\leq \|a\|^{2}_\infty\cdot  \sup_{\gamma\in K}\|h_i(\gamma) -1_{r(\gamma)}\|^2\to 0.
	\end{align*}

Now suppose conversely that   $(\widetilde{m}_{h_i})_{i \in \Ind}$ converges  pointwise on $C_c(\bu)\subseteq B_E$ to the identity operator, 
or that   $(h_i)_{i \in \Ind}=(h^{\Phi_i})_{i \in \Ind}$ where $(\Phi_i)_{i\in I}$  converges to identity pointwise on $C_c(\bu)\subseteq C_r^*(\bu)$ in the  $\|\cdot\|_E$ norm. Let $K\subseteq \G$ be a compact set. 
	We need to show that $
	\sup_{\gamma\in K}\|h(\gamma)-1_{r(\gamma)}\|\to 0$. There is a finite  open cover $\{V_j\}_{j=1}^n$ of $K$
 such that $\overline{V_j}$ is compact
	and contained in an open bisection $U_j\in S$, for each $j=1,...,n$.
	Accordingly, there are sections $f_j\in C_c(\B|_{U_j})$ such that
	$f_j(\gamma)=[1_{r(\gamma)}, c_{U_j}(\gamma)]$ for $\gamma\in \overline{V_j}$ and $j=1,...,n$.
	If $(h_i)_{i \in \Ind}=(h^{\Phi_i})_{i \in \Ind}$, then
	\begin{align*}
	\sup_{\gamma\in \overline{V_j}}\|h^{\Phi_i}(\gamma)-1_{r(\gamma)}\|
	&=\sup_{\gamma\in \overline{V_j}}\|\Phi_i(f_j)(\gamma) f_j(\gamma)^* -f_j(\gamma) f_j(\gamma)^*\|
	\\
	&\leq \sup_{\gamma\in \overline{V_j}}\| \Phi_i(f_j)(\gamma)- f_j(\gamma) \|\leq \|\Phi_i(f_j) -f_j\|_{E} \to 0.
	\end{align*}
	In the remaining case we have 
	$
	\sup_{\gamma\in \overline{V_j}}\|h_i(\gamma)-1_{r(\gamma)}\|
	=\sup_{\gamma\in \overline{V_j}}\|\widetilde{m}_{h_i}(f_j)(\gamma) -f_j(\gamma)\|
	\leq \|\widetilde{m}_{h_i}(f_j) -f_j\|_{E}\to 0.
	$
		Thus in both cases $
	\sup_{\gamma\in K}\|h(\gamma)-1_{r(\gamma)}\|\leq \max_{j=1,...,n} \sup_{\gamma\in \overline{V_j}}\|h(\gamma)-1_{r(\gamma)}\|\to 0$.
\end{proof}

\section{The Haagerup property for twisted groupoid dynamical systems}\label{Haagerup property}

In this section we analyse the relations between the geometric Haagerup property for group\-oids, based on conditions introduced in \cite{Tu}, and its $C^*$-algebraic counterpart from \cite{dong_ruan}, working in the context of general twisted groupoid $C^*$-dynamical systems. 
For second countable groups the Haagerup property can be formulated using various equivalent conditions, see \cite[Definition 1.1.1]{book}.
These conditions can be naturally generalised to \'etale topological groupoids. However, as we explain below, the relationship between some of them  is subtle when the space of units is not compact.

\begin{defn}[\cite{Tu}]\label{def:Tu}
 A	\emph{locally proper negative type function }on an  \'etale groupoid $\G$ is a continuous function $\psi: \G \to \R$ which is
	\begin{rlist}
		\item \emph{normalised}: $\psi|_{\Gz}=0$;
		\item \emph{symmetric}: $ \psi(g) = \psi(g^{-1})$ for all $g \in G$;
		\item 	\emph{conditionally negative-definite}, i.e.\ for each $x \in \Gz$, $n \in \N$, $g_1, \ldots g_n \in \G_x$ and $\lambda_1, \ldots, \lambda_n \in \R$ such that $\sum_{i=1}^n \lambda_i = 0$ we have
		\[ \sum_{i,j=1}^n \overline{\lambda_i}\psi(g_i g_j^{-1}) \lambda _j \leq 0,\]
		\item \emph{locally proper}: the function $(\psi, r, s): \G \to \R \times \Gz \times \Gz$ is proper.	
	\end{rlist}
\end{defn}
\begin{rem}\label{rem:Haagerup_affine_actions} Suppose that the groupoid $\G$ is $\sigma$-compact. 
 If $\G$ acts properly on a continuous field $H=(H_x)_{x\in \Gz}$ of affine Euclidean spaces, see \cite[Definition 3.2, 3.3]{Tu}, then the field admits a continuous section that allows us to write the action  as $\gamma \xi=\alpha_\gamma(\xi)+b(\gamma)$ where $\alpha$ is the linear part and $\gamma\to b(\gamma)\in H_{r(\gamma)}$ is a cocycle. 
Putting $\psi(\gamma):=\|b(\gamma)\|^2$ we obtain a locally proper negative type function. 
Moreover, every locally proper function of negative type has this form, see \cite[Proposition 3.8]{Tu}.
Hence, $\G$  admits a locally proper negative type function if and only if  $\G$ acts properly on a continuous field of affine Euclidean spaces.
In particular, this holds when $\G$ is amenable, by \cite[Lemma 3.5]{Tu}.
\end{rem}

It is standard that if $h: \G \to \C$ is positive-definite then for each $g \in \G$ we have  $h(g) = \overline{h(g^{-1})}, 
|h(g)| \leq h(r(g))$. Further if the function as above is real-valued, then to verify positive-definiteness it suffices to show the corresponding condition working only with real coefficients. Normalised symmetric conditionally negative-definite functions take only non-negative values. Finally if $\Gz$ is compact then a function $\psi:\G \to \R$ as above is locally proper if and only if it is proper in the usual sense (as a function from $\G$ to $\R$). 
We will make an extensive use of  countable exhaustions by compact sets, as discussed in the lemma below (which is likely well-known, but we could not locate a reference).

\begin{lem}\label{lem:rudin}
	A locally compact Hausdorff space  is $\sigma$-compact if and only if it admits a countable exhaustion by compact sets, i.e.\ there is a sequence $(L_n)_{n \in \N}$ of compact sets
	whose union is the whole space and  $L_n$ is contained in the interior of $L_{n+1}$, for each $n\in \N$.
\end{lem}
\begin{proof}
	The `if' part is clear. To show the `only if' part let $(K_n)_{n\in \N}$  be a sequence of compact sets  that cover the  space. 
	Put $L_1:=K_1$. By \cite[Theorem 2.7]{Rudin87}, for any open $U$ containing $L_1\cup K_2$ there is an open set $V$ with  compact closure such that
	$L_1\cup K_2\subseteq V\subseteq \overline{V}\subseteq U$. Put $L_2=\overline{V}$. Apply the same fact  to $L_2\cup K_3$ to find $L_3$, and proceed by induction.  
\end{proof}

The proof of the next fact 
uses standard ideas,  appearing in different setups for example in \cite{Clare} in the setting of measured groupoids, in  \cite{Jolissaint} for measured equivalence relations.
\begin{prop} \label{Haagequivgeneral}
	Let $\G$ be an \'etale locally compact  Hausdorff groupoid. Consider the following conditions: 
	\begin{enumerate}
	\item\label{enu:Haagequivgeneral1} $\G$ admits a locally proper negative type function;
	\item \label{enu:Haagequivgeneral2} 
	there exists a sequence of continuous positive-definite functions $(h_n)_{n \in \N}$ such that
	\begin{rlist}
		\item each $h_n$ is \emph{normalised}: $h_n|_{\Gz}=1$,
		\item each $h_n$ is \emph{locally $C_0$}: for each  compact set $K\subseteq \Gz$ 
		we have
		$h_n|_{\G_K^K}\in C_0(\G_K^K)$ where $\G_K^K:=r^{-1}(K)\cap s^{-1}(K)$,
		\item the sequence converges to $1$ uniformly on compact sets;
	\end{rlist}
\item\label{enu:Haagequivgeneral3}  there is a  net $(h_i)_{i \in \Ind}\subseteq  C_0(\G)$ of  positive-definite functions, with values in the unit disk and converging to $1$ uniformly on 
	compact subsets of $\G$.	
	\end{enumerate}
	Then \eqref{enu:Haagequivgeneral1}$\Rightarrow$\eqref{enu:Haagequivgeneral2}$\Rightarrow$\eqref{enu:Haagequivgeneral3}. 
	If $\G$ is $\sigma$-compact, then \eqref{enu:Haagequivgeneral1}$\Leftrightarrow$\eqref{enu:Haagequivgeneral2} and a net in \eqref{enu:Haagequivgeneral3} may be replaced by a sequence.
	If $\G$ is $\sigma$-compact and $\Gz$ is a  union of compact open sets, 
	then all the conditions \eqref{enu:Haagequivgeneral1},\eqref{enu:Haagequivgeneral2},\eqref{enu:Haagequivgeneral3} are equivalent. 
\end{prop}

\begin{proof}  \eqref{enu:Haagequivgeneral1}$\Rightarrow$\eqref{enu:Haagequivgeneral2}:
	Assume  that $\psi$ is a locally proper negative type function, and define 
	\[h_n (\gamma)= \exp\left(-\frac{1}{n} \psi(\gamma)\right), \;\;\;\; n \in \N, \gamma \in G.\]
	Then $h_n$ is continuous and  normalised. It is positive-definite by well-known arguments (here we use the fact it is real-valued).
	Moreover the sequence $(h_n)_{n \in \N}$  converges to $1$ (uniformly on compacts, as $\psi$ is continuous, so bounded on compact sets). 
	It remains to note that each $h_n$ is locally $C_0$: without loss of generality we consider $n=1$. Fix $\epsilon\in (0,1)$  and note that the pre-image of the set $[0,- \log \epsilon] \times K \times K$ with respect to $(\psi, r, s)$ is compact -- denote it by $L$. 
	Then for $\gamma \in \G$ such that $r(\gamma), s(\gamma) \in K$ we have that $|h_1(\gamma)|< \epsilon$  if and only if $\gamma \notin L$.
	
	\eqref{enu:Haagequivgeneral2}$\Rightarrow$\eqref{enu:Haagequivgeneral3}: Let $(h_n)_{n \in \N}$ be as described in \eqref{enu:Haagequivgeneral2}. Then it is bounded by $1$ as using positive-definiteness we have $|h(\gamma)|\leq |h(r(\gamma))|=1$ for $\gamma\in \G$.
	We can modify  $(h_n)_{n \in \N}$, dropping the normalisation condition, so that its members are $C_0$, and not just locally $C_0$. 
	To this end, we fix an approximate unit $(f_i)_{i\in I}\subseteq C_c(\Gz)^+$ for $C_0(\Gz)$.
	Then the net $((f_i\circ r) h_n (f_i\circ s))_{(n,i) \in \N\times I}$ satisfies the conditions in \eqref{enu:Haagequivgeneral3}; in particular each of the functions is positive-definite and is bounded by $1$.

From now on \emph{we assume $\G$ is $\sigma$-compact}. By Lemma \ref{lem:rudin}, $\G$  admits a countable exhaustion $(L_n)_{n \in \N}$ by compact sets.
	In particular, if $(h_i)_{i \in \Ind}$ is a net as in	\eqref{enu:Haagequivgeneral3} for each $n\in \N$ we may 
	choose $i_n\in \Ind$ so that $\|h_{i_n}-1_{L_n}\|_\infty <1/n$. Then the sequence $h_{i_n}$ converges to $1$ uniformly on 
	compact subsets of $\G$. Hence the net $(h_i)_{i \in \Ind}$ might be replaced by a sequence.
	
\eqref{enu:Haagequivgeneral2}$\Rightarrow$\eqref{enu:Haagequivgeneral1}:
 Fix two sequences of strictly positive numbers $(\alpha_n)_{n \in \N}$ and $(\epsilon_n)_{n \in \N}$ such that the first one is increasing to infinity, and the series $\sum_{n \in \N} \alpha_n \epsilon_n $ is convergent.
	Then proceed as follows: for each $n \in \N$ find a function $\phi_n:\G \to \C$ which is positive-definite, normalised, vanishes at infinity, and such that $|1-\phi_n(\gamma)|< \epsilon_n$ for $\gamma \in L_n$  (these are naturally just chosen by picking a subsequence of $(h_n)_{n \in \N}$). Then define the function $\psi:\G \to \R$ by the formula:
	\[ \psi(\gamma) = \sum_{n=1}^\infty \alpha_n \textup{Re} (1- \phi_n(\gamma)), \;\;\; \gamma \in \G.    \]
	The series above converges uniformly on compact subsets (as any of these is contained in almost all $L_n$), so yields a continuous function, which is conditionally negative-definite as a sum of  conditionally negative-definite  functions. It is obviously normalised.
	
	It remains to check that it is locally proper. Another exhaustion argument (for example) guarantees that any compact subset of 
	$\R \times \Gz \times \Gz$ is contained in a set of the form $[-M,M]\times K \times K$ for some $M >0$ and $K\subseteq \Gz$ compact. Fix then such $M$ and $K$. 
	Let $k \in \N$ be such that $\sum_{n=1}^k \alpha_n >2M$ and for each $n=1, \ldots, k$ find a compact set $L_n\subseteq \G$ such that for $\gamma \in \G\setminus L_n$ such that $r(\gamma), s(\gamma) \in K$ we have $|\phi_n(\gamma)| \leq \frac{1}{2}$. Then put $L = \bigcup_{n=1}^k L_n$. If $\gamma \in \G \setminus L$,  $r(\gamma), s(\gamma) \in K$, then
	\[ 
	\psi(\gamma) \geq  \sum_{n=1}^k \alpha_n \textup{Re} (1- \phi_n(\gamma)) \geq \sum_{n=1}^k \alpha_n  (1- |\phi_n(\gamma)|) \geq M.  
	\]
	Thus  the pre-image of the set $[-M, M] \times K \times K$ with respect to $(\psi, r, s)$ is contained in $L$. This ends the proof of \eqref{enu:Haagequivgeneral2}$\Rightarrow$\eqref{enu:Haagequivgeneral1}.

Assume now that  $\G$ is $\sigma$-compact and $\Gz$ is a  union of compact open sets. We will then show that \eqref{enu:Haagequivgeneral3}$\Rightarrow$\eqref{enu:Haagequivgeneral1}. 
		There is an exhaustion $(L_n)_{n \in \N}$ of $\G$ such that $r(L_n)=s(L_n)$ and this set, denoted by $K_n$, is compact open in $
	\Gz$, for all $n\in \N$. Indeed, our assumptions clearly imply that there is an exhaustion $(K_n)_{n \in \N}$ of $\Gz$ by compact open sets  and an exhaustion $(\tilde{L}_n)_{n \in \N}$ of $\G$ by compact sets.  Then it suffices to find an increasing sequence $(m_n)_{n \in \N}$ such that $K_n \subseteq \tilde{L}_{m_n}$ for each $n \in \N$ and put $L_n:=\tilde{L}_{m_n}\cap s^{-1}(K_n)\cap r^{-1}(K_n)$. 
	
	Let $(h_n)_{n=1}^\infty\subseteq  C_0(\G)$ be   an approximate unit	   consisting  of  positive-definite functions.
	We can modify this sequence so that for each $n\in \N$  the $n$-th element of the sequence is normalised on $K_n$. Indeed, we may find 
	an increasing sequence of  natural numbers $(k_n)_{n\in \N}$ such that $\|h_{k_n}|_{K_n} - 1_{K_n}\|< \frac{1}{2}$. Consider   the new functions defined by the formula
	\[ 
	\widetilde{h_n}(\gamma) = \begin{cases}
	(h_{k_n}(r(\gamma))^{-\frac{1}{2}}  h_{k_n}(\gamma) (h_{k_n}(s(\gamma))^{-\frac{1}{2}},&  \gamma \in s^{-1}(K_n)\cap r^{-1}(K_n),
	\\
	0, & \gamma \notin  s^{-1}(K_n)\cap r^{-1}(K_n).
		\end{cases}
	\]
		The function $\widetilde{h_n}$ is continuous, as $K_n$ is clopen in $\G$, and explicit verification shows it is positive-definite. 
		Thus the sequence  $(\widetilde{h_n})_{n =1}^\infty$
	satisfies all the requirements, including the normalization condition: $\widetilde{h_n}|_{K_n}=1$, $n\in \N$. 	
	Choose  sequences $(\alpha_n)_{n\in \N}$ and  $(\epsilon_n)_{n\in \N}$ as in the proof of  the implication \eqref{enu:Haagequivgeneral2}$\Rightarrow$\eqref{enu:Haagequivgeneral1}.
	By choosing a subsequence   of $(\widetilde{h_n})_{n =1}^\infty$  we get a sequence $(\phi_n)_{n\in \N}$ such that for
	each $n\in \N$,  $\phi_n\in C_0(\G_{K_{n}}^{K_{n}})$, $\phi_n|_{K_n}=1$ and $|1-\phi_n(\gamma)|< \epsilon_n$ for $\gamma \in L_n$. Then define the function $\psi:\G \to \R$ by the formula:
	\[ \psi(\gamma) = \sum_{n=1}^\infty \alpha_n \textup{Re} (1_{\G_{K_n}^{K_n}}(\gamma)- \phi_n(\gamma)), \;\;\; \gamma \in \G.    \]
	The series above as before converges uniformly on compact subsets so yields a continuous function, which is conditionally negative-definite as a sum of  conditionally negative-definite  functions; the fact that the individual factors are conditionally negative-definite can be verified directly, and uses the fact that each $\phi_n$ is supported on $\G_{K_n}^{K_n}$. The function $\psi$ is obviously normalised and symmetric. 
\end{proof}
\begin{rem}
It follows from the proof above that if $\G$ is $\sigma$-compact and $\Gz$ is a  union of compact open sets, 
then the net in \eqref{enu:Haagequivgeneral3}, if it exists, can be arranged so that $\sup_{i\in \Ind, x\in \Gz} |h_i(x)|\leq 1$.
\end{rem}
Let $\G$ be a $\sigma$-compact  \'etale  Hausdorff groupoid. Then $\Gz$ is a  union of compact open sets if and only if the $C^*$-inclusion 
$C_0(\Gz)\subseteq C_r^*(\G)$ is \emph{relative $\sigma$-unital} in the sense of  \cite[Definition 2.1]{Matsumoto}.
This condition is automatically satisfied when $\Gz$ is compact  or totally disconnected (equivalently, the inclusion is unital or the groupoid $\G$ is ample). 
We do not know whether this condition is necessary to deduce the equivalence between conditions \eqref{enu:Haagequivgeneral3} and \eqref{enu:Haagequivgeneral1}  in Proposition \ref{Haagequivgeneral}. On the other hand, we would like to view the Haagerup property as an approximation property. Therefore in this paper we  choose the condition \eqref{enu:Haagequivgeneral2} as the most convenient one.
It generalises to groupoid actions as follows.

	\begin{defn}\label{def:Haagerup_for_twisted_actions}
	Let $(\A,\G,\Sigma,\alpha)$ be a twisted groupoid  $C^*$-dynamical system
	such that the bundle $\A$   has a continuous unit section.
	We say that $(\A,\G,\Sigma,\alpha)$ has the \emph{Haagerup  property} 
	if 
	if there is a net $(h_i)_{i \in \Ind}$ of continuous positive-definite sections of $\A*_r\G$ such that
	\begin{rlist}
		\item\label{item:Haagerup_for_twisted_actions1} each $h_i$ is \emph{normalised}: $h_i|_{\Gz}$ is the unit section of $\A$,
		\item\label{item:Haagerup_for_twisted_actions2}
		each $h_i$ is \emph{locally $C_0$}: for each  compact set $K\subseteq \Gz$ 
		we have
		$h_i|_{\G^K_K} \in  C_0(\A|_K*_r\G^K_K)$,
		\item\label{item:Haagerup_for_twisted_actions3} the net converges to the unit section uniformly on compact sets: for every compact $ K\subseteq \G$ we have
	$\sup_{\gamma\in K} \|h_i(\gamma) - 1_{r(\gamma)}\| \stackrel{i \in \Ind}{\longrightarrow} 0$.
		\end{rlist}
The groupoid $G$ itself has the \emph{Haagerup property} if the above holds for the trivial bundle $\A=\Gz\times \C$ (i.e.\
 there is a net with properties listed in Proposition \ref{Haagequivgeneral} (\ref{enu:Haagequivgeneral2})).
\end{defn}

\begin{rem}\label{Haagerup for twisted actions} 
We could define the  Haagerup property above without the assumption that $\A$ has a continuous unit section by replacing the last condition
by the following one: for every compact $a\in C_c(\A*_r\G)$ we have
	$ \|h_ia - a\|_{\infty} \stackrel{i \in \Ind}{\longrightarrow} 0.$
However, this would itself imply that $A:=C_0(\A)$ has a central approximate unit. Indeed, by Proposition \ref{prop:pdCrossed} 
positive-definite sections $(h_i)_{i \in \Ind}$ are necessarily central, that is we necessarily have
$h_i(\gamma)\in Z(M(A_{r(\gamma)}))$ for each $\gamma$ and $i\in \Ind$. Thus 
$(m_{h_i}|_A)_{i \in \Ind}$ would be a central approximate unit in  $A$. 

One of the benefits from the continuous unit section assumption on  $\A$
is that we immediately obtain the desirable fact that if $\G$ has the Haagerup  property then every twisted $C^*$-dynamical systems  $(\A,\G,\Sigma,\alpha)$  the Haagerup  property. 
\end{rem}

\begin{prop}\label{prop:unital_untwisted_Haagerups}
A  $C^*$-dynamical system $(\A,\G,\alpha)$ with
$\Gz$ compact and  $A:=C_0(\A)$ unital has the Haagerup  property if and only if there is a net $(h_i)_{i \in \Ind}\subseteq C_0(\A*_r\G)$ of
	positive-definite sections  that  converge to the unit section  uniformly on compact sets.
\end{prop}
\begin{proof}
Since $\Gz$ is compact, being locally $C_0$ is the same as being $C_0$. Hence the ``only if'' part is trivial.
So  assume $(h_i)_{i \in \Ind}\subseteq C_0(\A*_r\G)$ is a net of
	positive-definite sections  that  converge to the unit section  uniformly on compact sets. 
	We need to  modify the elements of this net so that they are normalised. Since  $\Gz$ is compact this is  easy to observe: for big enough $i\in \Ind$, say $i \geq i_0$, we have that $\|h_i|_{\Gz} - 1_{\Gz}\|< \frac{1}{2}$. Consider  for $i\geq i_0$ the new functions defined by the formula
	\[ \widetilde{h_i}(\gamma) = (h_{i}(r(\gamma))^{-\frac{1}{2}}  h_{i}(\gamma) \alpha_{\gamma}\Big((h_{i}(s(\gamma))^{-\frac{1}{2}}\Big), \qquad \gamma \in \G.  \]
	The function $\widetilde{h_i}$ is continuous, as $\Gz$ is clopen in $\G$, and in fact $\widetilde{h_i}\in C_0(\A*_r\G)$.
	By construction $\widetilde{h_i}|_{\Gz}=1_{\Gz}$. 
	Since $h_i$ is positive-definite for 	
	 any finite set $F
\subseteq \G_x$, $x \in \Gz$, we have $\left[\alpha_{\gamma}^{-1} \left(h_i(\gamma\eta^{-1})\right)\right]_{\eta,\gamma \in F} \geq 0$ in $M_{|F|}(A_x)$. Putting $a_\gamma:=\alpha_{\gamma}^{-1}\big((h_{i}(r(\gamma))^{-\frac{1}{2}}\big)$, $\gamma\in F$ we get
positive elements in $A_x$ and  
$$
\left[\alpha_{\gamma}^{-1} \left(\widetilde{h_i}(\gamma\eta^{-1})\right)\right]_{\eta,\gamma \in F} 
=\left[a_\gamma\alpha_{\gamma}^{-1} \big(h_i(\gamma\eta^{-1})\big)   a_\eta\right]_{\eta,\gamma \in F} \geq 0.
$$
Hence $\widetilde{h_i}$ is positive-definite.  Let  $K\subseteq \G$ be a compact set.  The net   $(\widetilde{h_i}|_{K})_{i \in \Ind}$ 
 converges uniformly to the unit section $1_K$, as   the net $h_i$ 
converges uniformly to the unit section on  $K \cup r(K)\cup s(K)$, and therefore  
the nets of sections $K\in\gamma \to h_{i}(r(\gamma))^{-\frac{1}{2}}$ and  $K\in\gamma \to \alpha_{\gamma}\Big(h_{i}(s(\gamma))^{-\frac{1}{2}}\Big)$ converge uniformly to the unit section on $K$.
\end{proof}
\begin{rem}\label{rem:Haagerup for twisted actions2} 
By Proposition \ref{prop:unital_untwisted_Haagerups}, Definition \ref{def:Haagerup_for_twisted_actions} generalises 
the  Haagerup property introduced  in \cite[Section 3]{dong_ruan} for discrete group actions on unital $C^*$-algebras.
Indeed,  
if $\G=\Gamma$ is a discrete group and the twist is trivial, then the groupoid action $(\A,\G,\alpha)$ is nothing but a group action  
$\alpha:\Gamma \curvearrowright A$ on a unital $C^*$-algebra $A=\A_e$. This  action has the  the Haagerup  property  if and only if  there exists a net  $(h_i)_{i \in \Ind}$ of functions $h_i: \Gamma \to Z(A)$ which are positive-definite, in the sense of Section 3 of \cite{dong_ruan},
and such that $h_i \to 1$ pointwise on $\Gamma$. A similar interpretation is
 true for actions of discrete groups twisted by $\T$-valued cocycles: 
again our definition means that  there exists a net  $(h_i)_{i \in \Ind}$ of functions $h_i: \Gamma \to Z(A)$ which are positive-definite in a twisted sense (see for example Definition 4.2 of \cite{BedosConti2}) and converge pointwise to $1$. Note however that even for non-twisted actions our notion differs from the Haagerup property of actions of discrete groups on $C^*$-algebras as defined in \cite{mstt}. Roughly speaking, in \cite{mstt} the authors allowed arbitrary mapping $\B$-multipliers.
\end{rem}
\begin{lem}\label{lem:Haagerup for twisted groupoids} 
Let $(\G,\Sigma)$ be a twisted groupoid 
where $\G$ is \'etale locally compact Hausdorff.
Treated  as a twisted action on  $C_0(\G^{(0)})$, $(\G,\Sigma)$ has the  Haagerup  property if and only if $\G$ has the  Haagerup  property.
\end{lem}
\begin{proof}
We identify $(\G,\Sigma)$ with $(\A,\G,\Sigma,\alpha)$ where $\A=\Gz\times \C$, $\alpha_\gamma=\textup{id}$ and $u_x(z)=1$ for $\gamma\in \G$, $x\in \Gz$ and $z\in \T$. Hence $\A*_r\G\cong \G \times \C$ and therefore $C_0(\A*_r\G)\cong C_0(\G)$.  It is immediate that the respective notions of positive-definiteness coincide.
\end{proof}
We will now modify the definition  of the Haagerup property for a $C^*$-inclusion introduced in \cite{dong_ruan}. Recall that  
if  $E:B \to A\subseteq B$ is a faithful conditional expectation from a $C^*$-algebra $B$ onto $A$, then \(B_E\) denotes the right Hilbert $A$-module associated to $E$ via \eqref{eq:def_B_E}.
In fact $B_E$ is a $C^*$-correspondence over $A$ with natural left action. 
To define ``locally compact'' maps on $B_E$ we will use the notion of a \emph{Pedersen's ideal}, i.e.\  a minimal 
hereditary dense ideal (see \cite[5.6]{Pedersen}).
  Recall  that every $C^*$-algebra $A$ contains such an ideal; we denote it by $K(A)$. The ideal  $K(A)$ is a linear span of 
  $K(A)_+:=\{a\in A_+: a\leq \sum_{i=1}^{n}a_k, a_k\in K(A)_0\}$ where $K(A)_0:=\{f(a): a\in A_+, \,\, f\in C_c(0,\infty), f \geq 0\}$. 
  
\begin{lem} Let $\A$ be a $C^*$-bundle over a locally compact Hausdorff space $X$ that admits a continuous unit section. 
The Pedersen's ideal of $A=C^*(\A)$ is $C_c(\A)$. 
\end{lem}
\begin{proof} Recall that 
Since  $\A$ is unital, we may identify $C_c(X)$ with a subalgebra of $C_c(\A)\subseteq A$. 
Then obviously $C_c(X)_+\subseteq K(A)_0$ and every $a\in C_c(\A)_+$ is dominated by some element in $C_c(X)_+$.
Hence  $C_c(\A) \subseteq K(A)$. The reverse inclusion follows because $C_c(\A)$ is clearly a hereditary ideal in $A$.
\end{proof}

 \begin{defn}\label{defn:E-Haagerup}
Let $A \subseteq B$ be a nondegenerate $C^*$-inclusion  equipped with a faithful conditional expectation $E:B \to A$. 
 We say that $B$ has the \emph{$E$-Haagerup  property} 
 if there exists a  net $(\Phi_i)_{i \in \Ind}$
 of  completely positive maps on $B$ such that 
\begin{rlist}
		\item $\Phi_i|_A=id|_A$ for each $i \in \Ind$;
		\item 
		for each $i \in \Ind$, $\Phi_i$ extends to a bounded map  $\widetilde{\Phi}_i:B_E\to B_E$ which is ``locally compact'', that is 
		$k\widetilde{\Phi}_ik\in \K(B_E)$ for every $k\in K(A)$ (here $(k\widetilde{\Phi}_ik)(b):=k \widetilde{\Phi}_i(b)k$ for all $b \in B_E$);
		\item $\lim_{i \in \Ind}\|\Phi_i(b) - b\|_{B_E} = 0$ for every $b \in B$.
		\end{rlist}
 We  will say that a \emph{$C^*$-inclusion $A\subseteq B$ has the
	Haagerup  property} if the inclusion has the unique conditional expectation, denoted $E$, and $B$ has the $E$-Haagerup  property. 
\end{defn}

\begin{rem}\label{rem:Haagerup_for_inclusions}
Condition (i) above implies that $\Phi_i$ is   a contractive  $A$-bimodule map. Indeed, since we assume that $A$ and $B$ have a common (contractive) approximate unit $(\mu_\lambda)_{\lambda \in \Lambda}$ we get $\|\Phi\|=\lim_{\lambda\in \Lambda}\|\Phi_i(\mu_\lambda)\|=\lim_{\lambda\in \Lambda}\|\mu_\lambda\|=1$.
  Hence the multiplicative domain of $\Phi_i$ is equal to\
	$\{b\in B:\Phi_i(bb^*)=\Phi_i(b)\Phi_i(b^*),\,\, \Phi_i(b^*b)=\Phi_i(b^*)\Phi_i(b)\}$. 
	This together with (i) implies  that $\Phi_i$ is an $A$-bimodule map. 
 \end{rem}
\begin{rem}
A sufficient condition for a completely positive map $\Phi:B\to B$ to extend to a bounded map $\widetilde{\Phi}:B_E\to B_E$ with $\|\widetilde{\Phi}\|\leq m$ (where $m>0$), is that  $E\circ \Phi\leq m E$.
Indeed,  for every  $b\in B$ we then have
$$E(\Phi(b)^*\Phi(b))\leq \|\Phi\|E(\Phi(b^*b))\leq m \|\Phi\|E(b^*b),$$ 
cf. \cite{dong_ruan}. In fact, in the definition of the Haagerup property it is usually assumed (cf.\ for instance \cite{dong_ruan}, \cite{Dong}) that $E\circ \Phi\leq E$. We did not include it in our definition as in the context of twisted groupoid crossed products this can always be arranged, see Remark \ref{rem:on_sub_E_condition}.
\end{rem}

\begin{rem}
Dong and Ruan speak simply of the $A$-Haagerup property, but this is not quite precise, as shown by Suzuki in \cite{Suzuki}: in general the property depends not only on the pair $(A,B)$ but also on the choice of a conditional expectation $E$. On the other hand, there are natural situations when the inclusion $A\subseteq B$ has a  unique  conditional expectation. This holds  for instance when $B$ is the reduced cross-sectional $C^*$-algebra of a Fell bundle $\B$ over a topologically free \'etale locally compact Hausdorff groupoid $\G$, and more generally for
	 noncommutative Cartan subalgebras $A\subseteq B$, cf.\ \cite{Exel:noncomm.cartan}, \cite{BartoszRalf3}. 
\end{rem}

It is immediate that our definition is consistent with the one in \cite{Dong}. It is also consistent with the one given by  Dong and Ruan \cite{dong_ruan}, concerning unital inclusions, if we assume further  that the conditional expectation is tracial -- as is the case for the canonical expectation related to crossed products of actions of discrete groups on commutative algebras, being the main focus of \cite{dong_ruan}. 
\begin{prop}\label{prop:unital_Haagerups} Let $A \subseteq B$ be a unital inclusion of $C^*$-algebras and let 
$E:B \to A\subseteq B$ be a faithful tracial conditional expectation. 
Then $B$ has the $E$-Haagerup  property if and only  if there exists a  net $(\Phi_i)_{i \in \Ind}$
of	completely positive  $A$-bimodule maps on $B$ such that, 
for every $i \in \Ind$,  $\Phi_i$ extends to a compact map $\widetilde{\Phi}_i\in \K(B_E)$ and  $\lim_{i \in \Ind}\|\Phi_i(b) - b\|_{B_E} = 0$ for every $b \in B$.	
\end{prop}
\begin{proof} 
	
	Since $1\in A\subseteq B$ ``locally compact'' and ``compact'' on $B_E$ is the same.
Hence the ``only if'' part is clear by Remark \ref{rem:Haagerup_for_inclusions}. Let then $(\Phi_i)_{i \in \Ind}$ be 	completely positive  $A$-bimodule maps on $B$ such that for every $i \in \Ind$ we have
$\widetilde{\Phi}_i\in \K(B_E)$ and $\lim_{i \in \Ind}\|\Phi_i(b) - b\|_{B_E} = 0$ for  $b \in B$. Find 
 $i_0\in \Ind$ such that $\|\Phi_i(1) - 1\|< \frac{1}{2}$ for $i \geq i_0$. Note that $\Phi_i(1)$ is a positive element commuting with $A$ (as $\Phi_i$ is an $A$-bimodule map). Thus putting,  for $i \geq i_0$,
$$
\Psi_i(b):=\Phi_i(1)^{-\frac{1}{2}}\Phi_i(b) \Phi_i(1)^{-\frac{1}{2}}, \qquad b\in B, 
$$
we get a completely positive map $\Psi_i$   on $B$ such that 
 $\Psi_i|_A=id|_A$ and $\Psi_i$  extends to a compact map $\widetilde{\Psi}_i=\Phi_i(1)^{-\frac{1}{2}}\widetilde{\Phi}_i \Phi_i(1)^{-\frac{1}{2}}$ on $B_E$.
Moreover, for $b\in B$  we have $\lim_{i \geq i_0}\|\Psi_i(b) - b\|_{B_E} = 0$ 
by a $3\epsilon$-argument. Note that this is the place where we use the tracial property of $E$, so that we have the `right' estimate $\|bc\|_{B_E}\leq \|c\|_\infty \|b\|_{B_E}$ for $b, c \in B$.
\end{proof}

Using the facts from previous sections we get the following far reaching generalisation of  \cite[Theorem 3.6]{dong_ruan},  one of the main results of our paper.
\begin{thm}\label{FH}
Let $B:=C^*_r(\A,\G,\Sigma, \alpha)$ be the reduced crossed product of the twisted action 
 $(\A,\G,\Sigma,\alpha)$ of an \'etale locally compact Hausdorff groupoid $\G$ on 
the $C^*$-algebra $A:=C_0(\A)$, and let   $E:B \to A$ be the canonical  conditional expectation.
Then  $B$ has the $E$-Haagerup  property if and only if $(\A,\G,\Sigma,\alpha)$ has the  Haagerup  property.
\end{thm}

\begin{proof}
	Suppose that $(\A,\G,\Sigma,\alpha)$ has the Haagerup  property. Let $(h_i)_{i \in \Ind}$ be a net as in Definition \ref{def:Haagerup_for_twisted_actions}.
	By Proposition \ref{prop:positive-defnite_vs_completely_positive}
		we have the corresponding completely positive contractive multiplier  maps 
	\(m_{h_i}:C_r^*(\bu)\to C_r^*(\bu)\).
	By this proposition \(\Phi_i:=m_{h_i}\) is an  $A$-bimodule map.  
	By Definition \ref{def:Haagerup_for_twisted_actions} \eqref{item:Haagerup_for_twisted_actions1} 
	we have $\Phi_i|_A=id|_A$ and the form of $E$ implies that
	$E\circ \Phi_i=E$. Hence $(\widetilde{\Phi}_i)_{i \in \Ind}$ is uniformly bounded by $1$, and 
	therefore 	it 
converges pointwise to the identity operator on $B_E$, by Lemma \ref{lem:convergence_multipliers}.
	Finally, let $k\in K(A)=C_c(\A)$. Take any positive function $h\in C_c(\Gz)$ 
	which is equal to $1$ on the support of $k$. Then $\widetilde{h}_i(\gamma):=h(r(\gamma))h_i(\gamma)h(s(\gamma))$, $\gamma\in \G$, 
	is a positive-definite section of $\A*_r\G$. 
Denote the support of $h$ by $K$. Since $\widetilde{h}_i$ vanishes outside $\G_{K}^{K}$
 and $h_i|_{K}^{K}\in  C_0(\A|_{K}*_r\G^{K}_{K})$	(by Definition \ref{def:Haagerup_for_twisted_actions} \eqref{item:Haagerup_for_twisted_actions2}) 
	we get $\widetilde{h}_i\in C_0(\A*_r\G)$. Therefore   \(\widetilde{m}_{\widetilde{h}_i }\in \mathcal{K}(B_E)\)
	by Proposition \ref{prop:Hilbert multipliers_fell_groupoids} \eqref{enu:multipliers_fell_groupoids2}. 
	Hence $k\Phi_i k=k \widetilde{m}_{\widetilde{h}_i }k\in \mathcal{K}(B_E)$.

	Now let $(\Phi_i)_{i \in \Ind}$ be any net as in Definition 	\ref{defn:E-Haagerup}. Let $h^{\Phi_{i}}$, $i \in \Ind$, be given by 
	Proposition \ref{prop:Haagerup Trick}. The formula  \eqref{eq:Haagerup_multiplier} and the equality  $\Phi_i|_A=id|_A$ imply that  $h^{\Phi_{i}}|_{\Gz}=1_{\Gz}$. 
	By Lemma \ref{lem:convergence_multipliers},  $(h^{\Phi_{i}})_{i \in \Ind}$
	converges to the unit section uniformly on compact sets. To see that 
		$h_i|_{\G^K_K} \in  C_0(\A|_K*_r\G^K_K)$, for any   compact set $K\subseteq \Gz$, fix such a $K$ and
	take a positive function $h\in C_c(\Gz)$ equal to $1$ on $K$. Multiplying it by the unit section of $\A$ 
	we get  an element $k\in K(A)$. 
	It  follows  from \eqref{eq:Haagerup_multiplier} that $h^{k\Phi_ik}(\gamma)=h(r(\gamma))h^{\Phi_i}(\gamma)h(s(\gamma))$, $\gamma \in \G$.
	Since $k\Phi_ik\in \K(B_E)$, we get $h^{k\Phi_ik}\in C_0(\A*_r\Gz)$ by Proposition
	\ref{prop:Hilbert multipliers_fell_groupoids} \eqref{enu:multipliers_fell_groupoids4}.
	Therefore $h^{\Phi_i}|_{\G^K_K}=h^{k\Phi_ik}|_{\G^K_K}\in C_0(\A|_K*_r\G^K_K)$.
	This verifies that $(\A,\G,\Sigma,\alpha)$ has the Haagerup property.
\end{proof}
\begin{rem}\label{rem:on_sub_E_condition}
It follows from the proof above that $B:=C^*_r(\A,\G,\Sigma, \alpha)$ has the $E$-Haagerup  property if and only if we can find a net $(\Phi_i)_{i \in \Ind}$
as in Definition \ref{defn:E-Haagerup} with the additional property that $E\circ \Phi_i=E$ for each $i \in \Ind$.
\end{rem}

Theorem \ref{FH}, apart from generalising the main result of \cite{dong_ruan}, has an immediate consequence for the relative Haagerup property with respect to a Cartan subalgebra.

\begin{cor}\label{cor:Cartan_Haagerup}
Let  $A \subseteq B$ be a (commutative) Cartan subalgebra of a $C^*$-algebra $B$. Then $A\subseteq B$ has the 
Haagerup  property if and only if
the associated Weyl groupoid has the Haagerup   property. 
\end{cor}
\begin{proof}
We may assume identifications $A=C_0(\Gz)$ and $B=C_r^*(\G,\Sigma)$ for a twisted groupoid $(\G,\Sigma)$ (this is well known when $B$ is separable \cite{Re} but holds in general, see \cite{BartoszRalf3} or \cite{Raad}). The twisted groupoid $(\G,\Sigma)$  is uniquely determined by the
inclusion $A\subseteq B$, and $\G$ is called the Weyl groupoid of $A\subseteq B$. Thus it suffices to combine Theorem \ref{FH} and Lemma \ref{lem:Haagerup for twisted groupoids}.
\end{proof}

As another immediate consequence, we obtain a generalisation of \cite[Theorem~4.2]{dong_ruan}  from countable groups to metric spaces with bounded geometry.
\begin{cor}
	Let $X$ be a metric space with  bounded geometry and $\G(X)$ be its coarse groupoid. Then $X$ coarsely embeds into a Hilbert space if and only if  there is a  net $(h_i)_{i \in \Ind}\subseteq  C_0(\G(X))$ of  positive-definite functions,  converging to $1$ uniformly on 
	compact subsets of $\G(X)$ if and only  if  for all (equivalently, some) twist $\Sigma$ over $\G(X)$, the $C^*$-inclusion  $C(\G(X)^{(0)})\subseteq C_r^*(\G(X), \Sigma)$ has the Haagerup  property.
\end{cor}
\begin{proof}
	By \cite[Proposition 3.2]{STY}, $\G(X)$ is a locally
compact Hausdorff, principal \'etale groupoid. Moreover, $\G(X)$ is $\sigma$-compact and has a compact unit space. 
By \cite[Theorem 5.4]{STY}, $X$ coarsely embeds into a Hilbert space if and only if 
$G(X)$ admits a proper negative type function. Hence the assertion follows from Proposition \ref{Haagequivgeneral}, Lemma \ref{lem:Haagerup for twisted groupoids} and  Corollary~\ref{cor:Cartan_Haagerup}.
\end{proof}

\begin{rem} 
Recall that the coarse groupoid $G(X)$ of a bounded geometry metric space $X$ is amenable if and only if $X$ has property A if and only if $C_r^*(\G(X), \Sigma)$ is nuclear for all (equivalently, some) twist $\Sigma$ over $\G(X)$ (see \cite[Theorem~5.3]{STY} and \cite[Theorem 5.4]{Takeishi}). On the other hand, there exist metric spaces with bounded geometry which do not have property A but coarsely embed into a Hilbert space (see e.g. \cite{AGS, Osaj14}).

\end{rem}

\section{The Haagerup property and the Universal Coefficient Theorem} \label{Sect:UCT}

 In this section we discuss the consequences of the results of the last section for the questions related to the Universal Coefficient Theorem. We  recall that the Universal Coefficient Theorem (UCT) for $C^*$-algebras was introduced by Rosenberg and Schochet in \cite{RS87}. A separable $C^*$-algebra $A$ is said to satisfy the UCT if for every separable $C^*$-algebra $B$ the following natural sequence
$$
0\rightarrow \text{Ext}(K_*(A),K_{*-1}(B)) \rightarrow KK_*(A,B)\rightarrow \text{Hom}(K_*(A),K_*(B))\rightarrow 0
$$
is exact.  A separable $C^*$-algebra satisfies the UCT if and only if it is $KK$-equivalent to a commutative $C^*$-algebra (see \cite[Theorem~23.10.5]{B}). 
Although there exist exact non-nuclear $C^*$-algebras that do not satisfy the UCT (see \cite{S})\footnote{It is also worth noting that there exist separable non-exact $C^*$-algebras which satisfy the UCT. Indeed, this is the case for the reduced and full group $C^*$-algebras of $\Gamma$, where $\Gamma$ is a finitely generated non-exact group with the Haagerup property (as exhibited in \cite[Theorem~2]{Osaj14}).}, it is still open whether all separable nuclear $C^*$-algebras satisfy the UCT. This question is often refereed to as the UCT problem, and it is receiving renewed interest due to the recent breakthrough results in the classification program of separable simple nuclear $C^*$-algebras satisfying the UCT (see e.g.\ \cite{Win,Win2}).

It follows from Tu's remarkable paper \cite{Tu} that all $C^*$-algebras associated to groupoids with the Haagerup property satisfy the UCT. 
Building on Tu's techniques, Barlak and Li have proved that this also holds for twisted amenable \'etale groupoids (see \cite{BL}), using recent articles of Takeishi \cite{Takeishi} and of van Erp and Williams \cite{EW}. In this section, we generalise the key results of \cite{BL} concerning UCT, following the same idea. But instead of \cite{EW} we exploit  the Stabilization Theorem \cite{stabilization} of Ionescu, Kumjian, Sims and  Williams, and instead of amenability we assume the Haagerup property. This enables us to obtain  a stronger assertion under weaker assumptions.

 Recall that a Fell bundle is said to be \emph{separable }if each of its fibers is separable.
\begin{thm}\label{uct for crossed product}
Let $\bu$ be a continuous saturated and separable Fell bundle over a second countable locally compact Hausdorff \'etale groupoid $\G$. Assume 
that  $C^*(\bu|_{\Gz})$ is of type I and $\G$ has the Haagerup  property. Then both $C^*_r(\bu)$ and  $C^*(\bu)$ satisfy the UCT.

\end{thm}
\begin{proof}
We claim that we may reduce  the situation 
to the case where $\bu$
is the Fell bundle associated to a  groupoid dynamical system  $(\A,\G,\alpha)$, using the Stabilization Theorem \cite[Theorem 3.7]{stabilization}, see also  \cite{Lalonde}. 
Indeed, the Stabilization Theorem states that $\bu$ is Morita equivalent to a  groupoid $C^*$-dynamical system $(\A,\G,\alpha)$.
Then $A=C^*(\A)$ is Morita equivalent to $C^*(\bu|_{\Gz})$ and hence is of type I. Since $\B$ is separable and continuous,
$\A$ is separable and continuous as well (this is not stated explicitly in  \cite{stabilization} but it
follows immediately from the construction of the corresponding groupoid dynamical system).
By  Renault's equivalence theorems for Fell bundles, see  \cite[Theorem 14]{SimsWilliams}, 
and \cite[Theorem 6.4]{MW08},  $C^*_r(\bu)$ (resp. $C^*(\bu)$) is Morita equivalent to $C^*_r(\A, \G, \alpha)$ (resp. $C^*(\A, \G, \alpha)$), 
cf.\ also \cite{Lalonde}. Since the UCT is stable under Morita equivalence this shows our claim. 

Moreover, it follows from \cite[Proposition~4.12 and Theorem~9.3]{Tu} that $C^*(\A, \G, \alpha)$ and $C^*_r(\A, \G, \alpha)$ are $KK$-equivalent. Since the UCT is stable under $KK$-equivalences \cite[Theorem~23.10.5]{B}, it suffices to show that $C^*_r(\A, \G, \alpha)$
has the UCT.

So let us consider  a groupoid dynamical system $(\A,\G,\alpha)$ such that $\G$ is a second countable locally compact Hausdorff \'etale groupoid, $\A$ is a continuous bundle and  $A:=C_0(\A)$ is separable and type I.
 By Proposition \ref{Haagequivgeneral} and Remark \ref{rem:Haagerup_affine_actions} (which is essentially \cite[Proposition 3.8]{Tu}), $\G$ acts properly on a continuous field $H$ of affine Euclidean spaces. From this point on one can continue as in the proof of  
\cite[Theorem 3.1]{BL}. 
\end{proof}

In the next few results 
we consider 
twists for an action $\alpha$ of a \emph{discrete group} $\Gamma$ on a $C^*$-algebra $A$, associated to cocycles $\omega: \Gamma \times \Gamma \to UM(A)$. We refer to \cite{PackerRaeburn} for relevant definitions, and denote the respective reduced/universal crossed products by $A\rtimes_{\alpha,r}^\omega \Gamma$ and $A\rtimes_{\alpha}^\omega \Gamma$. Similarly, we write $C^*_r(\Gamma, \omega)$ and $C^*(\Gamma, \omega)$ for twisted reduced/universal group $C^*$-algebras (which  are crossed products by an action of $\Gamma$ on $\C$). In fact, the twisted group actions are particular examples of twisted inverse semigroup actions, as described in Definition \ref{defn:twisted_inverse_semigroup_action}, and the corresponding $C^*$-algebras coincide, cf. page 
\pageref{page:twisted_crossed_products}. In particular, twisted group actions can be viewed as saturated Fell bundles over groups. 
Hence Theorem~\ref{uct for crossed product} gives immediately the following generalisation of \cite[Proposition~6.1]{ELPW}.
\begin{cor}\label{typI+H}
Let $(A,\Gamma,\alpha,\omega)$ be a twisted $C^*$-dynamical system such that the $C^*$-algebra $A$ is separable and type I,  and $\Gamma$ is a countable discrete group with the Haagerup property. Then $A\rtimes_{\alpha,r}^\omega \Gamma$ and $A\rtimes_{\alpha}^\omega \Gamma$ satisfy the UCT. 
\end{cor}

We  list some more concrete applications.
\begin{cor}
Let $\Gamma$ be a countable discrete group. If $N$ is a virtually abelian normal subgroup of $\Gamma$ such that $\Gamma/N$ has the Haagerup property, then $C^*_{r}(\Gamma)$ and $C^*(\Gamma)$ satisfy the UCT.
\end{cor}
\begin{proof}
It is well-known that $C^*_{r}(\Gamma)\cong C^*_{r}(N)\rtimes^{\omega}_{\alpha,r}(\Gamma/N)$ for some twisted action $(\alpha,\omega)$ of the quotient group $\Gamma/N$ on $C^*_{r}(N)$ (see e.g.\ \cite[Theorem~4.1]{PackerRaeburn} and \cite[Theorem~2.1]{Be}). Since $N$ is virtually abelian, $C^*_{r}(N)$ is type I and we complete the proof by the previous proposition. The same argument is valid in the universal case.
\end{proof}
\begin{cor}\label{twisted semidirect}
Suppose that $H, N$ are countable discrete groups such that  $H$ acts on $N$ by automorphisms and assume that $\omega:N\times N \rightarrow \mathbb{T}$ is a 2-cocycle invariant under the $H$-action (i.e.\ $\omega(h\cdot n,h\cdot m)=\omega(n,m)$ for all $n,m\in N$ and $h\in H$), so that we have  natural actions of $H$ on  
$C^*_{r}(N,\omega)$ and on $C^*(N,\omega)$.

If both $N$ and $H$ have the Haagerup property, then  both $C^*_{r}(N,\omega)\rtimes_{r}H$ and  $C^*(N,\omega)\rtimes H$ satisfy the UCT.
\end{cor}
\begin{proof}
Since $H$ has the Haagerup property, $C^*_{r}(N,\omega)\rtimes H$ and $C^*_{r}(N,\omega)\rtimes_{r}H$ are $KK$-equivalent (as are $C^*(N,\omega)\rtimes H$ and $C^*(N,\omega)\rtimes_{r}H$). Hence, we only have to show the UCT for $C^*_{r}(N,\omega)\rtimes_{r}H$ (the proof for  $C^*(N,\omega)\rtimes_{r}H$ would be identical from this point). For this it suffices to show the UCT of $C^*_{r}(N,\omega)\rtimes_r F$ for every finite subgroup $F$ of $H$ (see \cite[Corollary~9.4]{MN06} and \cite{HK}). On the other hand, it follows from \cite[Lemma~2.1]{ELPW} that $$C^*_{r}(N,\omega)\rtimes_r F \cong C^*_{r}(N\rtimes F, \tilde{\omega}),$$
where $\tilde{\omega}$ is a 2-cocycle on $N\rtimes F$ given by $\tilde{\omega}((n,h),(n',h'))=\omega(n,h\cdot n')$, $h \in H, n, n' \in N$.

Since $N\rtimes F$ has the Haagerup property as well, $C^*_{r}(N\rtimes F, \tilde{\omega})$ satisfies the UCT for every finite subgroup $F$ by Corollary~\ref{typI+H}. This ends the proof.
\end{proof}
\begin{ex}
Let $\theta$ be an irrational number. We will identify $\theta$ with the real $2\times 2$ skew-symmetric matrix 
$\begin{pmatrix}
0 & -\theta \\
\theta & 0
\end{pmatrix}
$.
Define a 2-cocycle $\omega_\theta:\Z^2\times \Z^2\rightarrow \T$ by $\omega_\theta(x,y):=e^{- \pi i\langle\,\theta x,y\rangle}$, $x,y \in \Z^2$.  Then the reduced twisted group $C^*$-algebra $C_{r}^*(\Z^2,\omega_\theta)$ is known to be isomorphic to $A_\theta$, the irrational rotation algebra. Moreover, $\omega_\theta$ is invariant under the action of $SL(2,\Z)$ on $\Z^2$ via matrix multiplication (see \cite[Section~2]{ELPW} for details). From Corollary~\ref{twisted semidirect} we see that $A_\theta\rtimes_{r}SL(2,\Z)=C_{r}^*(\Z^2,\omega_\theta)\rtimes_{r}SL(2,\Z)$ satisfies the UCT (as does its universal counterpart). It is worth noticing that we can not apply Corollary~\ref{typI+H} directly here, as $A_\theta$ is not a type I $C^*$-algebra, and $A_\theta\rtimes_{r}SL(2,\Z)=C_{r}^*(\Z^2\rtimes SL(2,\Z),\tilde{\omega}_\theta)$ but the group $\Z^2\rtimes SL(2,\Z)$ does not have the Haagerup property.  
\end{ex}

We will now  combine Theorem \ref{uct for crossed product}  with the key results of Section 5. 
We start with a generalisation of  \cite[Theorem 1.1]{BL}.

\begin{prop}\label{uct for twisted}
	Assume that $(\G,\Sigma)$ is a twisted \'etale Hausdorff locally compact second countable groupoid. If $C_r^*(\G,\Sigma)$ has the $E$-Haagerup  property with respect to the canonical conditional expectation $E:C_r^*(\G,\Sigma)\rightarrow C_0(\G^{(0)})$ (which is automatic when $C_r^*(\G,\Sigma)$ is nuclear), then both $C_r^*(G,\Sigma)$  and $C^*(G,\Sigma)$ satisfy the UCT.
\end{prop}
\begin{proof}
	Theorem \ref{FH} and Lemma \ref{lem:Haagerup for twisted groupoids}
	imply that $C_r^*(\G,\Sigma)$ has the $E$-Haagerup  property if and only if $\G$ has the  Haagerup  property.
	When $C_r^*(G,\Sigma)$ is nuclear, then $\G$ is amenable by \cite[Theorem 5.4]{Takeishi} and hence has the Haagerup property by Remark \ref{rem:Haagerup_affine_actions}. 
	If $\G$ has the  Haagerup  property, then treating $(\G,\Sigma)$ as a continuous Fell line bundle over $\G$,  cf.\ Remark \ref{rem:line_bundles_twists},  get  that
	$C_r^*(G,\Sigma)$  and $C^*(G,\Sigma)$ satisfy the UCT by Theorem~\ref{uct for crossed product}.  
\end{proof}

We now present some applications of Proposition \ref{uct for twisted} involving Cartan subalgebras. The first one generalises \cite[Proposition~3.4 and Corollary~3.2]{BL}.

\begin{cor}\label{GammaCartan pair}
	Let $(B,A)$ be a separable  Cartan pair, and let $\Gamma$ be a countable group acting on $B$ such that $\gamma(A)=A$ for every $\gamma\in \Gamma$.  If $ B\rtimes_r \Gamma$ has the the $E$-Haagerup property with respect to the canonical conditional expectation  $E:B\rtimes_r \Gamma \rightarrow  A$ factorising via $B$ (which is automatic when $ B\rtimes_r \Gamma$ is nuclear), then $B\rtimes_r\Gamma$ satisfies the UCT. 
	
	In particular, if $B$ is a   separable $C^*$-algebra which admits a  Cartan subalgebra $A$ and the inclusion $A\subseteq B$ has the Haagerup  property, then $B$ satisfies the UCT.
\end{cor}
\begin{proof}
	By \cite[Theorem~5.9]{Re} the  pair $(B,A)$ is isomorphic to a pair $(C^*_r(\G,\Sigma)),C_0(\Gz))$
	where $(\G,\Sigma)$ is a twisted \'etale Hausdorff locally compact second countable groupoid.
	In the proof of \cite[Proposition~3.4]{BL} it is shown that $(\Gamma\ltimes \G,\Gamma\ltimes \Sigma)$ is also a twisted \'etale Hausdorff locally compact second countable groupoid and $(C_r^*(\Gamma\ltimes \G,\Gamma\ltimes \Sigma), C_0(\G^{(0)}) )\cong (B\rtimes_r \Gamma, A)$. Hence,  $B\rtimes_r\Gamma$ satisfies the UCT by Proposition~\ref{uct for twisted}.	
	The second statement follows by putting $\Gamma=\{e\}$.
\end{proof}

The second corollary extends  \cite[Corollary~7.4]{BFPR19}.
\begin{cor}
Let $B$ be a  separable $C^*$-algebra and let $A$ be an  abelian subalgebra of $B$. If there exists a discrete abelian group $\Gamma$ such that $(B,A)$ is a $\Gamma$-Cartan pair as defined in \cite[Definition~3.10]{BFPR19} and $B$ has the $\Delta$-Haagerup  property with respect to the canonical conditional expectation $\Delta: B \rightarrow A$, then $B$ satisfies the UCT . 
\end{cor}

\begin{proof}
It follows from \cite[Theorem~4.36]{BFPR19} that $(B,A)\cong (C_r^*(\G,\Sigma),C_0(\G^{(0)}))$ for a twisted \'etale Hausdorff locally compact second countable groupoid $(\G,\Sigma)$. Hence, $B\cong C_r^*(\G,\Sigma)$ satisfies the UCT by Proposition~\ref{uct for twisted}. 
\end{proof}

We finish by a remark concerning the case where we can deduce the UCT property for all the quotients of a given $C^*$-algebra.
The following corollary applies  to the reduced groupoid $C^*$-algebra  of a second countable locally compact Hausdorff 
\'etale groupoid   which is residually topologically principal and inner exact, see for instance \cite[Corollary 3.12]{BL18}.
\begin{cor} 
	Let $A\subseteq B$ be a Cartan inclusion where $B$ is separable and 
	$A$ separates ideals in $B$, i.e.\ the map $J\mapsto J\cap A$ defined on ideals of $B$ is injective. 
	If the inclusion $A\subseteq B$ has the Haagerup  property then all quotients of $B$ satisfy  the UCT.
\end{cor}
\begin{proof} 
	By \cite{Re} and Remark \ref{rem:line_bundles_twists} we may identify $(B,A)$ with $(C^*_r(\LL),C_0(\Gz))$
	for a continuous Fell line bundle $\LL$ over $\G$. 
	If  $C_0(\Gz)$ separates the ideals of $C_r^*(\LL)$, then every ideal in $C_r^*(\LL)$ is generated by 
	a $\G$-invariant ideal in $C_0(\Gz)$, see for instance \cite[Propositions  2.11, 6.9]{BartoszRalf4}. 
	Hence every ideal in $C_r^*(\LL)$ is generated by   $C_0(\Gz\setminus D)$ where $D$ is a closed \(\G\)-invariant set.
	This implies that the sequence 
	$$
	0 \to C_0(\LL|_{X\setminus D}) \to C_r^*(\LL) \to C_r^*(\LL|_D)\to 0, 
	$$
	which exists by \cite[Propositions 4.2, 4.3]{Lalonde}, is exact. 
	Thus every quotient of $C_r^*(\LL)$ is of the form $C_r^*(\LL|_D)$.	
	If $C_0(\Gz)\subseteq C^*_r(\LL)$ has the Haagerup  property, then by Theorem~\ref{FH}, Lemma~\ref{lem:Haagerup for twisted groupoids} and Proposition~\ref{Haagequivgeneral}, $G$ has the Haagerup  property. So does the closed subgroupoid $G|_D$ and the UCT of $C_r^*(G|_D,\Sigma|_D)$ follows from Theorem~\ref{uct for crossed product}.  
\end{proof}

\subsection*{Acknowledgements}
AS was partially supported by the National Science Center (NCN) grant no.~2014/14/E/ST1/00525. KL has received funding from the European Research Council (ERC) under the European Union's Horizon 2020 research and innovation programme (grant agreement no. 677120-INDEX). BKK was supported by the National Science Center (NCN) grant no.~2019/35/B/ST1/02684. We thank the referee for a careful reading of our manuscript and many useful comments.


\begin{thebibliography}{fulllllllll}

	\bibitem[A-D]{Clare}C.\,Anantharaman-Delaroche, \emph{The Haagerup property for discrete measured groupoids}, Operator algebra and dynamics, 1–-30, Springer Proc. Math. Stat., \textbf{58}, Springer, Heidelberg, 2013.
	
	

	\bibitem[AGS]{AGS} G.\,Arzhantseva, E.\,Guentner and J.\,\^Spakula, \emph{Coarse non-amenability and coarse embeddings},
Geom. Funct. Anal. \textbf{22} (2012), no.\,1, 22--36. 
	
	\bibitem[BaL]{BL}S.\,Barlak and X.\,Li, \emph{Cartan subalgebras and the UCT problem}, Adv. Math. \textbf{316} (2017), 748--769.
	
	\bibitem[Bla]{B}B.\,Blackadar, ``$K$-theory for operator algebras'', Cambridge University Press, Second edition (1998).
	
	\bibitem[B\'ed]{Be}E. B\'edos, \emph{Discrete groups and simple $C^*$-algebras}, Math. Proc. Cambridge Philos. Soc. \textbf{109} (1991), no. 3, 521--537. 

\bibitem[BC$_1$]{BedosConti}
\'E.\,B\'edos and R.\,Conti, \emph{Fourier series and twisted \cst-crossed products}, J. Fourier Anal. Appl. \textbf{21} (2015), no.\,1, 32--75. 

 
 
 \bibitem[BC$_2$]{BedosConti2}
 \'E.\,B\'edos and R.\,Conti, \emph{The Fourier-Stieltjes algebra of a $C^*$-dynamical system}, Internat. J. Math. 
 \textbf{27} (2016), no.\,6, 1650050, 50 pp.
 


\bibitem[BiF]{Bice}
T. Bice and I. Farah, \emph{Traces, ultrapowers, and the Pedersen-Petersen $C^*$-algebras}, Houston J. Math. \textbf{41} (2015), 1175--1190.



\bibitem[B\"oL]{BL18}
 C. B\"{o}nicke and K. Li, \emph{Ideal structure and pure infiniteness of ample groupoid $C^*$-algebras}, Ergodic Theory Dynam. Systems \textbf{40} (2020), no. 1, 34--63.



\bibitem[BFPR]{BFPR19}
J. H. Brown, A.H. Fuller, D. R. Pitts and S. A. Reznikoff, \emph{Graded {$C^*$}-algebras and twisted groupoid {$C^*$}-algebras}, 
New York J. Math. \textbf{27} (2021), 205--252.

\bibitem[BrO]{BO}
 N.\ P.\ Brown and N.\ Ozawa,  ``{$C^*$}-algebras and finite dimensional approximations'', {\rm American Mathematical Society, 2008}.
 
 \bibitem[BuS]{BusbySmith}
R. C. Busby and H. A. Smith, \emph{Representations of twisted group algebras}, Trans. Amer. Math. Soc. \textbf{149} (1970), 503--537.

 \bibitem[BuE$_1$]{BussExel0}
A. Buss and R. Exel, \emph{Twisted actions and regular Fell bundles over inverse semigroups},
Proc. Lond. Math. Soc. (3) \textbf{103} (2011), no. 2, 235--270.

 \bibitem[BuE$_2$]{BussExel}
 A. Buss and R. Exel, \emph{Fell bundles over inverse semigroups and twisted \'etale groupoids}, J.
 Operator Theory \textbf{67} (2012), no. 1, 153--205. 
 
 \bibitem[BEM]{BussExelMeyer}
A. Buss, R. Exel and  R. Meyer, \emph{Reduced $C^*$-algebras of Fell bundles over inverse
semigroups}, Israel J. Math. \textbf{220} (2017), no. 1, 225--274.

 \bibitem[BuM]{BussMeyer}
A. Buss and R. Meyer, \emph{Inverse semigroup actions on groupoids}, Rocky Mountain J.
Math. \textbf{47} (2017), no. 1, 53--159.




\bibitem[CCJJV]{book} P.A.\,Cherix, M.\,Cowling, P.\,Jolissaint, P.\,Julg and A.\,Valette, ``Groups with the Haagerup property.
Gromov's a-T-menability'',  Progress in Mathematics, 197,  Basel, 2001.

\bibitem[Cho]{choda} M. Choda, \emph{Group factors of the {H}aagerup type},
Proc. Japan Acad. Ser. A Math. Sci.
\textbf{59} (1983), 174--177.





\bibitem[DFSW]{dfsw} M.\,Daws, P.\,Fima, A.\,Skalski and S.\,White, \emph{The Haagerup property for locally compact quantum groups}, J.\,Reine\,Angew.\,Math.\, \textbf{711} (2016), 189--229.


\bibitem[Don]{Dong} Z.\,Dong, \emph{Haagerup property for $C^*$-algebras}, J.\,Math.\,Anal.\,Appl. \textbf{377} (2011), no.\,2, 631--644.

\bibitem[DoR]{dong_ruan}
Z.\,Dong and Z.J.\,Ruan, \emph{A Hilbert module approach to the Haagerup property}, Integr. Equ. Oper. Theory, {\bf 73} (2012), 431--454.

\bibitem[DuG]{dupre_gillette}
M. J. Dupr\'e and R. M. Gillette, ``Banach Bundles, Banach Modules and Automorphisms of
$C^*$-Algebras'', Res. Notes Math., vol. 92, Pitman (Adv. Publ. Program), Boston 1983.

\bibitem[ELPW]{ELPW} 
S. Echterhoff, W. L\"uck, N.C. Phillips and S. Walters, \emph{The structure of crossed products of irrational rotation algebras by finite subgroups of $SL_2(Z)$}, J. Reine Angew. Math. \textbf{639} (2010), 173-221.




\bibitem[ErW]{EW} E. van Erp and D. P. Williams, \emph{Groupoid crossed products of continuous-trace {$C^\ast$}-algebras}, J. Operator Theory \textbf{72} (2014), no. 2, 557-576.




\bibitem[Exe]{Exel:noncomm.cartan} 
R. Exel, \emph{Noncommutative Cartan subalgebras of $C^*$-algebras}, New York J. Math. \textbf{17} (2011),
331--382.
\bibitem[FeD]{Fell_Doran}
J. M. G. Fell and R. S. Doran, ``Representations of $*$-Algebras, Locally Compact Groups, and
Banach $*$-Algebraic Bundles'', Vol. 1. Basic Representation Theory of Groups and Algebras,
Pure Appl. Math., vol. \textbf{126}, Academic Press Inc., Boston, MA 1988.


\bibitem[Gre]{Green}
P. Green, \emph{The local structure of twisted covariance algebras}, Acta Math. \textbf{140} (1978),
no. 3-4, 191--250.


\bibitem[Haa$_1$]{Haa-metric}
U. Haagerup, \emph{An example of a non nuclear $C^*$-algebra, which has the metric approximation property}, Invent. Math. \textbf{50} (1978), no. 3, 279--293.


\bibitem[Haa$_2$]{Haagerup}
U. Haagerup, \emph{Group $C^*$-algebras without the completely bounded approximation
property}, J.\,of Lie Theory \textbf{26} (2016), no.\,3, 861--887.

\bibitem[HiK]{HK}
N. Higson and G. Kasparov, \emph{E-theory and KK-theory for groups which act properly and isometrically on Hilbert space}, Invent. Math. \textbf{144} (2001), no. 1, 23--74.


\bibitem[IKSW]{stabilization}
M. Ionescu, A. Kumjian, A. Sims and  D. P. Williams, \emph{A stabilization theorem
for Fell bundles over groupoids},  Proceedings of the Royal Society of Edinburgh: Section A Mathematics, \textbf{148} (2018), no.\,1, 79--100.






\bibitem[KrR]{KrausRuan} J.\,Kraus and Z.-J.\,Ruan, \emph{Approximation properties for Kac algebras}, Indiana Univ.\,Math.\,J. \textbf{48} (1999), no.\,2, 469--535.


\bibitem[Jol]{Jolissaint}
P. Jolisaint, \emph{The Haagerup property for measure-preserving standard equivalence relations}, 
	Ergodic Theory Dynam. Systems \textbf{25} (2005), no. 1, 161--174.

\bibitem[Kum$_1$]{Kumjian0} A. Kumjian, \emph{On $C^*$-diagonals}, Canad. J. Math. \textbf{38} (1986), 969--1008.
\bibitem[Kum$_2$]{Kumjian}
A. Kumjian, \emph{Fell bundles over groupoids}, Proc. Amer. Math. Soc. \textbf{126} (1998), no.\,4, 
1115--1125.



\bibitem[KwM$_1$]{BartoszRalf2} B. K. Kwa\'sniewski and R. Meyer, \emph{Essential crossed products by inverse semigroup actions: simplicity and pure infiniteness},  Doc. Math. \textbf{26} (2021), 271--335.



\bibitem[KwM$_2$]{BartoszRalf3} B. K. Kwa\'sniewski and R. Meyer, \emph{Noncommutative Cartan C*-subalgebras}, Trans. Amer. Math. Soc.   \textbf{373} (2020), no. 12, 8697--8724.

\bibitem[KwM$_3$]{BartoszRalf4} B. K. Kwa\'sniewski and R. Meyer, \emph{Stone duality and quasi-orbit spaces for generalised $C^*$-inclusions}, Proc.\,Lond.\,Math.\,Soc.  \textbf{121} (2020), no. 4, 788--827.

\bibitem[Lal
]{Lalonde}
S. Lalonde, \emph{Some consequences of stabilization theorem for Fell bundles over exact groupoids}, J. Operator Theory \textbf{81} (2019), no.\,2, 335--369.


\bibitem[Lan]{Lance}
E.C.Lance,  ``Hilbert C$^{*}$-modules. A toolkit for operator algebraists'', London Mathematical Society Lecture Note Series, {\bf 210}, Cambridge University Press, Cambridge, 1995.



\bibitem[Laz]{Lazar}
A. J. Lazar, \emph{A selection theorem for Banach bundles and applications}, J. Math. Anal. Appl.\textbf{ 462} (2018), 448--470.



\bibitem[Li]{XinLi} X.\,Li, Every classifiable simple $C^*$-algebra has a Cartan subalgebra, \emph{Invent.\,Math.} \textbf{219 }(2020), no.\,2, 653--699.


\bibitem[Mat]{Matsumoto}
K. Matsumoto, \emph{Relative Morita equivalence of Cuntz-Krieger algebras and flow equivalence of topological Markov shifts}, 
Trans. Amer. Math. Soc. \textbf{370} (2018), 7011--7050.




\bibitem[MSTT]{mstt}
A. McKee, A. Skalski, I. G. Todorov and L. Turowska,
\emph{Positive Herz-Schur multipliers and approximation properties of crossed products}, Math.\,Proc.\,Camb.\,Phil.\,Soc. \textbf{165} (2018), no. 3, 511--532.

\bibitem[MTT]{mtt}
A. McKee, I. G. Todorov and L. Turowska,
\emph{Herz-Schur multipliers of dynamical systems}, Adv.\,Math. \textbf{331} (2018), 387--438.

\bibitem[MeN]{MN06} R. Meyer and R. Nest
\emph{The Baum-Connes conjecture via localisation of categories}, Topology \textbf{45} (2006), no. 2, 209--259.

\bibitem[MuW]{MW08}
P. Muhly and D. P. Williams, \emph{Equivalence and disintegration theorems for Fell bundles
and their $C^*$-algebras}, Dissertationes Math. (Rozprawy Mat.) \textbf{456} (2008), 1--57.


\bibitem[Mur]{Mur}  G. J. Murphy
\emph{Positive definite kernels and Hilbert $C^*$-modules}
Proc. Edinburgh Math. Soc., \textbf{40} (1997),  367-374.

\bibitem[Osa]{Osaj14} D. Osajda, \emph{Small cancellation labellings of some infinite graphs and applications}, Acta Math. 225 (2020), no. 1, 159--191.

\bibitem[PaR]{PackerRaeburn} 
J. A. Packer and  I. Raeburn, \emph{Twisted crossed products of $C^*$-algebras}, Math. Proc. Cambridge Phil. Soc. \textbf{106} (1989), 293--311.

\bibitem[Ped]{Pedersen} 
G. K. Pedersen, ``$C^*$-Algebras and Their Automorphism Groups,'' Academic Press, New York (1979).

\bibitem[Raa]{Raad} A.I.\,Raad, \emph{A Generalization of Renault's Theorem for Cartan Subalgebras}, Proc. AMS, \emph{to appear}, available at arXiv:2101.03265v2.


\bibitem[RaW]{RamsayWalter} A.\,Ramsay and  M.\,Walter, 
\emph{Fourier algebra of locally compact  groupoids},
 J. Funct. Anal. \textbf{148}(1997), 314--367.

\bibitem[RaT]{RaeburnThompson} I.\,Raeburn and  S.J.\,Thompson, 
\emph{Countably generated Hilbert modules, the Kasparov stabilisation theorem, and frames with Hilbert modules}, 
	Proc.\,Amer.\,Math.\,Soc. \textbf{131} (2003), no. 5, 1557--1564.


\bibitem[Ren$_1$]{Renault0} J.\,Renault, "A groupoid approach to $C^*$-algebras", Lecture Notes in Mathematics, No.
793, Springer-Verlag, Berlin-New York, 1980.


\bibitem[Ren$_2$]{Renault}
J.\,Renault,  \emph{Repr\'esentation des produits crois\'es d'alg\'ebres de groupoides}, J. Operator Theory
\textbf{18}  (1987), 67--97.

\bibitem[Ren$_3$]{Renault2}
J.\,Renault,  \emph{The ideal structure of groupoid crossed product $C^*$-algebras}, J. Operator Theory \textbf{25} (1991), no. 1, 3--36.

\bibitem[Ren$_4$]{RenaultFourier} J.\,Renault,  \emph{The Fourier algebra of a measured groupoid and its multipliers}, J. Funct. Anal. \textbf{145} (1997), 455--490.

\bibitem[Ren$_5$]{Re} J.\,Renault,  \emph{Cartan subalgebras in {$C^\ast$}-algebras}, Irish Math. Soc. Bull. \textbf{61} (2008), 29--63.

\bibitem[Ren$_6$]{RenaultAFA} J.\,Renault,  \emph{Groupoid cocycles and derivations}, Ann. Funct. Anal. \textbf{3} (2012), no.\,1, 1--20.



\bibitem[RoS]{RS87}
J.\,Rosenberg and C.\,Schochet,
 \emph{The K\"{u}nneth theorem and the universal coefficient theorem for Kasparov's generalised $K$-functor}, Duke Math. J. \textbf{55} (1987), 431--474.

\bibitem[Rud]{Rudin87} W.\,Rudin, ``Real and complex analysis'', McGraw-Hill, New York, 1987.

\bibitem[Sim]{Sims} 
A.\,Sims, \emph{Hausdorff \'etale groupoids and their $C^*$-algebras}, in ``Operator Algebras and Dynamics: Groupoids, Crossed Products, and Rokhlin Dimension'',  Advanced Courses in Mathematics - CRM Barcelona (2020), pp.59--120, available at arXiv:1710.10897v2.

\bibitem[SiW]{SimsWilliams}
A.\,Sims and D.P.\,Williams, \emph{An equivalence theorem for reduced Fell bundle $C^*$-algebras}, New York J.\,Math. \textbf{19}
(2013), 159--178.

\bibitem[Ska]{S}  G.\,Skandalis, \emph{Une notion de nucl\'earit\'e en K-th\'eorie (d'apr\`es J. Cuntz)},  K-Theory \textbf{1} (1988), no.\,6, 549--573.

\bibitem[STY]{STY}  G.\,Skandalis, J.\,L.\,Tu, and G.\,Yu, \emph{The coarse Baum-Connes conjecture and groupoids},  Topology \textbf{41} (2002), no.\,4, 807--834.

\bibitem[Suz]{Suzuki} Y.\,Suzuki, \emph{Haagerup property for $C^*$-algebras and rigidity of $C^*$--algebras with property (T)},  J.\,Funct.\,Anal. \textbf{265} (2013), no. 8, 1778--1799.


\bibitem[Tak]{Takeishi}
T.\,Takeishi, \emph{On nuclearity of $C^*$-algebras of Fell bundles over \'etale groupoids}, Publ. RIMS \textbf{50} (2014), no.\,2  , 251--268.


\bibitem[Tu]{Tu} J. L. Tu, \emph{La conjecture de Baum-Connes pour les feuilletages moyennables}, $K$-Theory \textbf{17} (1999), no. 3, 215-264.




\bibitem[Wil]{WLbook}
D.\,P.\,Williams, ``Crossed Products of $C^*$-Algebras,'' Mathematical Surveys and
Monographs \textbf{134}, American Mathematical Society, Providence RI, 2007.




\bibitem[Win$_1$]{Win}
W.\,Winter, \emph{Structure of nuclear C*-algebras: from quasidiagonality to classification and back again}, Proceedings of the ICM 2018, Rio de Janeiro, Vol.\,\textbf{2}, 1797--1820, 2018.

\bibitem[Win$_2$]{Win2}
W.\,Winter, \emph{QDQ vs. UCT},
In Abel Symposia \textbf{12}: “Operator Algebras and Applications: The Abel Symposium 2015”, 321--342, Springer, 2016.

\bibitem[Z-M]{Zeller-Meier} 
G.\,Zeller-Meier, \emph{Produits crois\'es d'une $C^*$-alg\'ebre par un groupe d'automorphismes}, J. Math. Pures et Appl.\  \textbf{47}  (1968),  101--239.









\end{thebibliography}
\end{document}